\documentclass{amsart}
\usepackage{amssymb,amsfonts,amsmath,amsthm,graphicx,mathrsfs,mathtools,xypic,hyperref,caption,svg,epstopdf}
\hypersetup{colorlinks=true,linkcolor=blue,citecolor=blue}
\newtheorem{Theorem}{Theorem}[section]
\newtheorem{Lemma}[Theorem]{Lemma}
\newtheorem{Proposition}[Theorem]{Proposition}
\newtheorem{Corollary}[Theorem]{Corollary}
\theoremstyle{definition}

\theoremstyle{remark}
\newtheorem*{Remark}{Remark}

\newcommand{\eps}{\epsilon}

\renewcommand{\S}{\mathscr{S}}

\newcommand{\dist}{\mathrm{dist}}
\newcommand{\diam}{\mathrm{diam}}

\newcommand{\post}{\mathrm{post}}

\DeclareMathAlphabet{\mathpzc}{OT1}{pzc}{m}{it}
\begin{document}
\title{Quasisymmetric Embeddability of Weak Tangents}
\author{Wen-Bo Li}
\address{Department of Mathematics, University of Toronto, 40 St. George Street, Toronto, Ontario, Canada M5S 2E4}
\email{wenboli@math.toronto.edu}
\date{}
\begin{abstract}
	In this paper, we study the quasisymmetric embeddability of weak tangents of metric spaces. We first show that quasisymmetric embeddability is hereditary, i.e., if $X$ can be quasisymmetrically embedded into $Y$, then every weak tangent of $X$ can be quasisymmetrically embedded into some weak tangent of $Y$, given that $X$ is proper and doubling. However, the converse is not true in general; we will illustrate this with several counterexamples. In special situations, we are able to show that the embeddability of weak tangents implies global or local embeddability of the ambient space. Finally, we apply our results to Gromov hyperbolic groups and visual spheres of expanding Thurston maps.
\end{abstract}
\maketitle

\tableofcontents

\section{Introduction}\label{Introduction}

The Gromov-Hausdorff distance gives a precise meaning to how close or far apart two arbitrary (compact) metric spaces are. The definition that is widely used nowadays can be traced back to Gromov in \cite{Gr81a} and \cite{Gr81b}. More precisely, the Gromov-Hausdorff distance of two compact metric spaces is the infimum Hausdorff distance of their isometric images in the same space. Intuitively, it measures how far the two compact metric spaces are from being isometric.

The weak tangents of a metric space are analogous to the tangent planes of a surface. Let $X$ be a metric space and $p \in X$. ``Blowing up" $X$ near $p$ generates a sequence of dilations, which provides a better and better illustration of the local behavior near $p$. A weak tangent at $p$ is the limit (if it exists) of such a sequence, where the limit is in the (pointed) Gromov-Hausdorff sense. Gromov's compactness theorem shows that a subconvergent metric space always exists, given some conditions on the ambient space. See Propositions \ref{precompactness} and \ref{psprecompactness}.

A natural question raised regarding weak tangents is: Is there any analytical or geometrical relation between the ambient space and its weak tangents? There are some answers to this question. One remarkable achievement in answering this question is the Cheeger-Colding theory. The Cheeger-Colding theory investigates the analytical and geometrical properties of the weak tangents of complete connected manifolds with lower bounded Ricci curvatures. See \cite{CC97},\cite{CC00a} and \cite{CC00b} for more details. The Cheeger-Colding theory is widely used in many important works, including the proof of Thurston's geometrization conjecture and the proof of the existence of K\"ahler-Einstein metrics on Fano manifolds. Weak tangents also play a role in studying Poincar\'e inequalities on metric spaces, see \cite{Ch99} and \cite{CK15} for examples. In \cite{BKM09} and \cite{BM13}, metric fractals are studied together with weak tangents.

A homeomorphism $f: X \to Y$ is \emph{$\eta$-quasisymmetric}, where $\eta:[0, \infty) \to [0, \infty)$ is a homeomorphism, if
\[
\frac{d_Y(f(x),f(y))}{d_Y(f(x),f(z))} \leq \eta \left(\frac{d_X(x,y)}{d_X(x,z)} \right)
\]
for all $x, y, z \in X$ with $x \neq z$. A quasisymmetry is a generalization of a bi-Lipschitz map, and it preserves the approximate shape and the relative size. A metric space is \emph{proper} if the closure of every ball is compact. A metric space is \emph{doubling} if every ball can be covered by a uniform finite number of balls with half radius. See Section \ref{Preliminaries} for more information.

In this paper, we study the quasisymmetric embeddability of weak tangents of metric spaces.  We denote by $WT_p(X)$ the collections of all weak tangents at $p$ of $X$. We are interested in how the quasisymmetric embeddability of the ambient space relates to that of its weak tangents.

\begin{Theorem}\label{QSweakT}
	Let $X, Y$ be proper, doubling metric spaces and $f: (X, p, d_X) \to (Y, q, d_Y)$ be an $\eta$-quasisymmetric map. For any weak tangent $T_pX \in WT_p(X)$, there exists a weak tangent $T_qY \in WT_q(Y)$ such that $T_pX$ is $\eta$-quasisymmetric equivalent to $T_qY$.
\end{Theorem}

Theorem \ref{QSweakT} shows that quasisymmetric embeddability is hereditary. Roughly speaking, any quasisymmetric embedding between two metric spaces induces quasisymmetric embeddings between their weak tangents.  However, the converse implication is not true in general. See Section \ref{CE} for several counterexamples. 

A metric space $X$ is self-quasisymmetric if for any point $p$ and any number $r > 0$ there exists a neighborhood of $p$ with diameter less than $r$ that is uniformly quasisymmetric to $X$, and quasi-self-symmetric if the neighborhoods above are uniformly quasisymmetric to some subsets of $X$ with bounded sizes. See Section \ref{LEIGE} for more details. In special situations, we are able to prove the following two partial converse implications:

\begin{Theorem}\label{sqe}
	Let $X$ be a proper, doubling, $\eta$-self-quasisymmetric space and $p$ be a point in $X$. If every weak tangent of $p$ is $\theta$-quasisymmetrically embedded into $Y$, then $X$ is $\vartheta$-quasisymmetrically embedded into $Y$, where $\vartheta$ depends only on $\eta$ and $\theta$.
\end{Theorem}

Theorem \ref{sqe} shows that self-quasisymmetric metric spaces satisfy a global converse implication. We can further strengthen Theorem \ref{sqe} by adding more conditions. See Propositions \ref{lsqp} and the remark following it.

The following local embedding theorem gives another type of converse implication:

\begin{Theorem}\label{wttoe}
	Let $X$ be a compact, proper, doubling, $\eta$-quasi-self-symmetric space, $p$ be a point in $X$ and $T_pX$ be a weak tangent in $WT_p(X)$. If $T_pX$ is $\theta$-quasisymmetrically embedded into Y, then there exists a ball $B(q, r)$ in $X$ such that $B(q,r)$ is $\vartheta$-quasisymmetrically embedded into $Y$, where $\vartheta$ depends only on $\eta$ and $\theta$. If $X$ is uniformly perfect, $r$ depends only on $X$ and $\eta$.
\end{Theorem}

Theorem \ref{wttoe} shows that quasi-self-symmetric spaces satisfy a local converse implication. Theorem \ref{wttoe} is a special case of Theorem \ref{wttob}, which induces several applications. 

A Kleinian group $\Gamma$ is a discrete subgroup of isometries of a hyperbolic space and the limit set of $\Gamma$ is the set of accumulation points of the orbit $\Gamma p$ of any element $p$ in the hyperbolic space. A Schottky set is a compact subset of $\mathbb{S}^2$ whose complement is a union of at least three round open disks whose closures are disjoint. See Section \ref{GHSG} for more details.

The first application is a rigidity theorem for Kleinian groups whose limit sets are Schottky sets. 

\begin{Theorem}\label{lsigm}
	Suppose that $\Gamma$ and $\tilde{\Gamma}$ are Kleinian groups whose limit sets $S$ and $\tilde{S}$ are Schottky sets, respectively.  We assume that $\Gamma$ acts on $S$ and $\tilde{\Gamma}$ acts on $\tilde{S}$ as uniform convergence groups. Let $p$, $\tilde{p}$ be points in $S$ and $\tilde{S}$, respectively. If there exist a weak tangent $T_pS$ at $p$ of $S$, a weak tangent $T_{\tilde{p}} \tilde{S}$ at $\tilde{p}$ of $\tilde{S}$ and a quasisymmetry $f: (T_pS, p_\infty) \to (T_{\tilde{p}}\tilde{S}, \tilde{p}_\infty)$ such that $f(p_\infty) = \tilde{p}_\infty$, then there exists a M\"obius transformation mapping $S$ to $\tilde{S}$.
\end{Theorem}

In simple words, Theorem \ref{lsigm} shows that a pointed quasisymmetric equivalence between any two weak tangents of the above spaces implies an equivalence of M\"obius transformations between the ambient spaces. Theorem \ref{lsigm} generalizes Theorem $1.1$ in \cite{Me14} to the situation of weak tangents.

An expanding Thurston map $f: S^2 \to S^2$ is a postcritically-finite branched covering map on a topological sphere with $\deg(f) \geq 2$ where $f$ locally expands $S^2$. A visual metric $\rho$ is a specific metric on $S^2$ generated by $f$ such that $f$ locally expands $(S^2, \rho)$ in a uniform way. We call $(S^2, \rho)$ a visual sphere of $f$. See Section \ref{VSETM} for a precise and rigorous  definition.

The other application is about the visual spheres of expanding Thurston maps:

\begin{Theorem}\label{vms}
	Let $(S^2, \rho)$ be a visual sphere of an expanding Thurston map $f : S^2 \to S^2$ that does not have periodic critical points. The following statements are equivalent:
	\begin{enumerate}
		\item $(S^2, \rho)$ is a quasi-sphere.
		\item Every weak tangent of $(S^2, \rho)$ is quasisymmetric to $\mathbb{R}^2$.
		\item There exists a weak tangent of $(S^2, \rho)$ that is quasisymmetric to $\mathbb{R}^2$.
		\item There exists an open subset $U \subset S^2$ such that $(U, \rho)$ is quasisymmetrically embedded into $\mathbb{R}^2$.
	\end{enumerate}
\end{Theorem}

Theorem \ref{vms} illustrates a complete characterization of when visual spheres of expanding Thurston maps are quasi-spheres. Wu also proved it in \cite{Wu19} with ideas from dynamics. Here we give an alternate proof with the ideas generated in Section \ref{LEIGE} of this paper. Lemma \ref{gqs} is essential in the proof of Theorem \ref{vms} and it may evoke independent interests for other problems.

The paper is organized as follows. In Section \ref{Preliminaries} we give basic definitions and properties of Gromov-Hausdorff distance and weak tangents. In Section \ref{QEIH} we prove Theorem \ref{QSweakT} and show some results inspired by this theorem. However, the converse is not true in general and  we illustrate several counterexamples in Section \ref{CE}. Section \ref{LEIGE} is devoted to proving Theorem \ref{sqe} and Theorem \ref{wttoe} and thus showing that in specific situations the embeddability of weak tangents implies the embeddability(or local embeddability) of the ambient space. In Section \ref{GHSG} we study Gromov hyperbolic spaces and groups and prove Theorem \ref{lsigm}. In the last section, Section \ref{VSETM}, we investigate expanding Thurston maps and visual spheres and prove Theorem \ref{vms}.

\section*{Acknowledgments}

The author wants to thank Sergiy Merenkov and Jeremy Tyson for their helpful advice. He is grateful to Ilya Kapovich for suggesting applications on hyperbolic groups and for an idea used in Section \ref{GHSG}. He thanks Mario Bonk and Angela Wu for explaining visual spheres of expanding Thurston maps.

\section{Preliminaries}\label{Preliminaries}

\subsection{Notations and preliminaries}

Let $(X,d_X)$ be a metric space. We denote by
\[
B(x,r) = \{x' \in X: d_X(x,x') < r\} \qquad \textrm{and} \qquad \overline{B}(x,r) = \{x' \in X: d_X(x,x') \leq r\}
\]
the open and closed balls centered at $a$ with radius $r$, respectively. Furthermore, we denote by
\[
\partial B(x,r) = \overline{B}(x,r) \setminus B(x,r) = \{x' \in X: d_X(x,x') = r\}
\]
the metric boundary of $B(x,r)$. Notice that $\partial B(x,r)$ is not the same as the topological boundary of $B(x,r)$.

 If $X'$ is a subset of $X$, we denote by
\[
N_r(X') = \{x \in X: \exists \ x' \in X' \ \textrm{such that} \ d_X(x,x') < r\}
\]
the $r$-neighborhood of $X'$ for any $r > 0$.

We denote by $\mathbb{S}^n, \mathbb{D}^n, \mathbb{H}^n$ the $n$-dimensional unit sphere, unit ball and hyperbolic space, respectively. Specifically, we denote by $| \ \cdot \ |$ the standard metric in Euclidean spaces.

A homeomorphism $f: X\to Y$ is \emph{(metrically) $H$-quasiconformal} for some $1\leq H <\infty$ if
\begin{align*}
\limsup_{r \to 0}\frac{\sup \{ d_Y(f(x),f(y)) : d_X(x,y) \leq r\}}{\inf \ \{ d_Y(f(x),f(y)) : d_X(x,y) \geq r\}} \leq H
\end{align*}
for $\forall x \in X$. A map is quasiconformal if it is $H$-quasiconformal for some $H$.

A homeomorphism $f: X \to Y$ is \emph{$\eta$-quasisymmetric}, where $\eta:[0, \infty) \to [0, \infty)$ is a homeomorphism, if
\[
\frac{d_Y(f(x),f(y))}{d_Y(f(x),f(z))} \leq \eta \left(\frac{d_X(x,y)}{d_X(x,z)} \right)
\]
for all $x, y, z \in X$ with $x \neq z$. The map $f$ is called \emph{quasisymmetric} if it is $\eta$-quasisymmetric for some \emph{distortion function} $\eta$.

It is clear that every quasisymmetry is quasiconformal.

Here are some useful properties of quasisymmetric maps, which will be used repeatedly in the paper. Readers may refer to Proposition $10.6$ and $10.8$ of \cite{He01} for a proof.

\begin{Proposition}\label{diam}
	Suppose $f:X\to Y$ and $g:Y\to Z$ are $\eta$ and $\theta$-quasisymmetric mappings, respectively.
	\begin{enumerate}
		\item The composition $g \circ f : X \to Z$ is an $\theta \circ \eta$-quasisymmetric map.
		\item The inverse $f^{-1}:Y\to X$ is an $\eta'$-quasisymmetric map, where $\eta'(t) = 1/\eta^{-1}\!\!\left(\frac{1}{t}\right)$.
		\item If $A$ and $B$ are bounded subsets of $X$ and $A\subset B$, then
		\begin{equation*}
		\frac{1}{2\eta\left( \frac{\diam(B)}{\diam(A)} \right)} \leq \frac{\diam(f(A))}{\diam(f(B))} \leq \eta\left( \frac{2\diam(A)}{\diam(B)} \right).
		\end{equation*}
	\end{enumerate}
\end{Proposition}

A metric space is called \emph{proper} if the closure of every ball is compact.

A metric space is called \emph{doubling} if there exists a universal constant $C \geq 1$ such that every subset of diameter $d$ can be covered by at most $C$ subsets of diameter at most $d/2$. Notice that a space is doubling implies that it is separable.

A metric space $X$ is called \emph{uniformly perfect} if there exists a universal constant $C \geq 1$ such that for each $x \in X$ and for each $r >0$, $B(x,r) \setminus \overline{B}(x,r/C)$ is nonempty whenever $X \setminus B(x,r)$ is nonempty. Uniform perfectness forbids isolated islands in a uniform manner. In other words, a metric space is uniformly perfect if for every empty metric annulus in the space, the ratio between the outer and inner radii is bounded from above.

A metric space is equipped with the \emph{intrinsic metric} if the distance between any two points equals the infimum length of all the paths joining them. A metric space whose metric is intrinsic is called a \emph{length space}. A \emph{geodesic} in a length space is a curve which is locally a distance minimizer(i.e., a shortest path).

We prove the following lemma in order to ``glue" quasisymmetries.

\begin{Lemma}\label{gqs}
	Let $(X, d_X)$,$(Y, d_Y)$ be two bounded length spaces and $f: X \to Y$ be a homeomorphism. Let $X = X_1 \cup X_2$ where $X_1$ and $X_2$ are closed subsets of $X$ such that $X_1 \cap X_2$ is connected, $\diam(X_1 \cap X_2) \neq 0$ and $f|_{X_1}$, $f|_{X_2}$ are $\eta$-quasisymmetries, then $f$ is an $\eta_1$-quasisymmetry where $\eta_1$ depends on $\eta$, $\diam(X_1)$, $\diam(X_2)$ and $\diam(X_1 \cap X_2)$.
\end{Lemma}

\begin{proof}
	We denote by
	\[
	\lambda \! : = \! \frac{1}{6}\min \! \left\{ \! \! \frac{\diam(X_1 \cap X_2)}{\diam(X_1)}, \! \frac{\diam(X_1 \cap X_2)}{\diam(X_2)}, \! \frac{\diam(f(X_1 \cap X_2))}{\diam(f(X_1))}, \! \frac{\diam(f(X_1 \cap X_2))}{\diam(f(X_2))} \!\! \right\}.
	\]
	
	Let $x,y,z$ be distinct points in $X$, then it is sufficient to prove that
	\[
	\frac{d_Y(f(x), f(y))}{d_Y(f(x), f(z))} \leq \eta_1 \left(\frac{d_X(x,y)}{d_X(x,z)} \right).
	\]
	We split the proof into three situations.
	
	Let $x, z \in X_1$ and $y \in X_2$. Since $X$ is a length space, $X = X_1 \cup X_2$ and $X_1,X_2$ are closed, there exists a point $w \in X_1 \cap X_2$ such that $d_X(x,y) = d_X(x,w) + d_X(w,y)$. 
	
	If $d_Y(f(x), f(w)) \geq \lambda \cdot d_Y(f(x), f(y))$, then
	\[
	\frac{d_Y(f(x), f(y))}{d_Y(f(x), f(z))} \leq \frac{1}{\lambda}\frac{d_Y(f(x), f(w))}{d_Y(f(x), f(z))} \leq \frac{1}{\lambda}\eta \left(\frac{d_X(x,w)}{d_X(x,z)} \right) \leq \frac{1}{\lambda}\eta \left(\frac{d_X(x,y)}{d_X(x,z)} \right).
	\]
	
	If $d_Y(f(x), f(w)) < \lambda \cdot d_Y(f(x), f(y))$, we choose a point $v$ on $X_1 \cap X_2$ such that $d_Y(f(y), f(w)) \geq d_Y(f(v), f(w)) \geq 3 \lambda d_Y(f(y), f(w))$. Notice that such a point exists due to the connectedness of $X_1 \cap X_2$ and the definition of $\lambda$. 
	
	Since $d_Y(f(x), f(w)) < \lambda \cdot d_Y(f(x), f(y))$, we have $(1-\lambda)d_Y(f(x), f(y)) \leq d_Y(f(y), f(w))$. Then
	\[
	\frac{d_Y(f(x), f(y))}{d_Y(f(x), f(z))} \leq \frac{1}{1-\lambda} \frac{d_Y(f(y), f(w))}{d_Y(f(x), f(z))} \leq \frac{1}{3 \lambda(1-\lambda)} \frac{d_Y(f(v), f(w))}{d_Y(f(x), f(z))}.
	\]
	
	Since $d_Y(f(x), f(w)) < \lambda \cdot d_Y(f(x), f(y))$ and $d_Y(f(v), f(w)) \geq 3 \lambda d_Y(f(y), f(w))$, we have $d_Y(f(v), f(w)) \geq 3 \lambda d_Y(f(y), f(w)) \geq 3\lambda\left(d_Y(f(x), f(y)) - d_Y(f(x), f(w))\right) \geq 3d_Y(f(x), f(w)) - 3\lambda d_Y(f(x), f(w)) = (3 - 3\lambda) d_Y(f(x), f(w))$. Then $d_Y(f(x), f(v)) \geq d_Y(f(v), f(w)) - d_Y(f(x), f(w)) \geq d_Y(f(v), f(w)) - 1/(3-3\lambda) \cdot d_Y(f(v), f(w)) \geq  1/2d_Y(f(v), f(w))$, and
	\begin{eqnarray*}
		\frac{d_Y(f(x), f(y))}{d_Y(f(x), f(z))} &\leq& \frac{2}{3 \lambda(1-\lambda)} \frac{d_Y(f(x), f(v))}{d_Y(f(x), f(z))} 
		\\
		&\leq& \frac{2}{3 \lambda(1-\lambda)} \eta \left(\frac{d_X(x,v)}{d_X(x,z)} \right) 
		\\
		&\leq& \frac{2}{3 \lambda(1-\lambda)}\eta \left(\frac{d_X(x,w) + d_X(w,v)}{d_X(x,z)} \right).
	\end{eqnarray*}
	
	Since
	\[
	\frac{d_Y(f(y), f(w))}{d_Y(f(v), f(w))} \leq \eta \left(\frac{d_X(y,w)}{d_X(v,w)} \right),
	\]
	we have $d_X(v,w) \leq 1/\eta^{-1}(1) \cdot d_X(y,w)$. We denote by $C : = \max\{1, 1/\eta^{-1}(1)\}$.
	
	Thus
	\begin{eqnarray*}
		\frac{d_Y(f(x), f(y))}{d_Y(f(x), f(z))} &\leq& \frac{2}{3 \lambda(1-\lambda)} \eta \left(\frac{d_X(x,w) + d_X(w,v)}{d_X(x,z)} \right) 
		\\
		&\leq& \frac{2}{3 \lambda(1-\lambda)} \eta \left(C\frac{d_X(x,w) + d_X(y,w)}{d_X(x,z)} \right) 
		\\
		&=& \frac{2}{3 \lambda(1-\lambda)} \eta \left(C\frac{d_X(x,y)}{d_X(x,z)} \right).
	\end{eqnarray*}
	
	Let $x,y \in X_1$ and $z \in X_2$. Similarly, there exists a point $w' \in X_1 \cap X_2$ such that $d_Y(f(x), f(z)) = d_Y(f(x), f(w')) + d_Y(f(w'), f(z))$. We also split the proof into two cases: $d_X(x,w') \geq \lambda \cdot d_X(x,z)$ or not. The proof of this situation is essentially the same as above.
	
	If $d_X(x,w') \geq \lambda \cdot d_X(x,z)$, then
	\[
	\frac{d_X(x,y)}{d_X(x,z)} \geq \lambda\frac{d_X(x,y)}{d_X(x,w')} \geq \lambda \eta^{-1} \left( \frac{d_Y(f(x), f(y))}{d_Y(f(x), f(w'))} \right) \geq \lambda \eta^{-1} \left( \frac{d_Y(f(x), f(y))}{d_Y(f(x), f(z))} \right).
	\]
	
	If $d_X(x,w') < \lambda \cdot d_X(x,z)$, we choose a point $v'$ on $X_1 \cap X_2$ such that $d_X(z,w') \geq d_X(v',w') \geq 3\lambda \cdot d_X(z,w')$. Notice that such a point exists due to the connectedness of $X_1 \cap X_2$ and the definition of $\lambda$. Then
	
	\[
	\frac{d_X(x,y)}{d_X(x,z)} \geq (1-\lambda) \cdot \frac{d_X(x,y)}{d_X(z,w')} \geq (3\lambda-3\lambda^2) \cdot \frac{d_X(x,y)}{d_X(v',w')}.
	\]
	
	Since $d_X(x,v') \geq d_X(v',w') - d_X(x,w') \geq d_X(v',w') - 1/(3-3\lambda) \cdot d_X(v',w') \geq  1/2d_X(v',w')$. Then
	\begin{eqnarray*}
		\frac{d_X(x,y)}{d_X(x,z)} &\geq& \frac{3\lambda-3\lambda^2}{2}\frac{d_X(x,y)}{d_X(x,v')}
		\\
		&\geq& \frac{3\lambda-3\lambda^2}{2} \eta^{-1} \left( \frac{d_Y(f(x), f(y))}{d_Y(f(x), f(v'))} \right) 
		\\
		&\geq& \frac{3\lambda-3\lambda^2}{2} \eta^{-1} \left( \frac{d_Y(f(x), f(y))}{d_Y(f(x), f(w'))+d_Y(f(w'), f(v'))} \right).
	\end{eqnarray*}
	
	Since
	\[
	\frac{d_Y(f(v'), f(w'))}{d_Y(f(z), f(w'))} \leq \eta \left(\frac{d_X(v',w')}{d_X(z,w')} \right),
	\]
	$d_Y(f(v'), f(w')) \leq \eta(1) \cdot d_Y(f(z), f(w'))$. We denote by $C' : = \max\{1, \eta(1)\}$.
	
	Thus
	\begin{eqnarray*}
		\frac{d_X(x,y)}{d_X(x,z)} &\geq& \frac{3\lambda-3\lambda^2}{2} \eta^{-1} \left( \frac{d_Y(f(x), f(y))}{d_Y(f(x), f(w'))+d_Y(f(w'), f(v'))} \right)
		\\
		&\geq& \frac{3\lambda-3\lambda^2}{2} \eta^{-1} \left(\frac{1}{C'} \frac{d_Y(f(x), f(y))}{d_Y(f(x), f(w'))+d_Y(f(z), f(w'))} \right)
		\\
		&=& \frac{3\lambda-3\lambda^2}{2} \eta^{-1} \left(\frac{1}{C'} \frac{d_Y(f(x), f(y))}{d_Y(f(x), f(z))} \right).
	\end{eqnarray*}
	
	We denote by $\eta_1(t) := C'  \eta \left(\frac{2}{3 \lambda(1-\lambda)} t \right)$, so it is the distortion function in this situation. 
	
	Let $x \in X_1$ and $y,z \in X_2$. The proof of this situation relies on the above results and it is verbatim the same. Similarly, let $w \in X_1 \cap X_2$ where $d_X(x,y) = d_X(x,w) + d_X(w,y)$ and $v \in X_1 \cap X_2$ where $d_Y(f(y), f(w)) \geq d_Y(f(v), f(w)) \geq 3 \lambda d_Y(f(y), f(w))$.
	
	If $d_Y(f(x), f(w)) \geq \lambda \cdot d_Y(f(x), f(y))$, then
	\[
	\frac{d_Y(f(x), f(y))}{d_Y(f(x), f(z))} \leq \frac{1}{\lambda}\frac{d_Y(f(x), f(w))}{d_Y(f(x), f(z))} \leq \frac{1}{\lambda}\eta_1 \left(\frac{d_X(x,w)}{d_X(x,z)} \right) \leq \frac{1}{\lambda}\eta_1 \left(\frac{d_X(x,y)}{d_X(x,z)} \right)
	\]
	where $\eta_1$ is the distortion function in the second situation.
	
	If $d_Y(f(x), f(w)) < \lambda \cdot d_Y(f(x), f(y))$, then
	\begin{eqnarray*}
	\frac{d_Y(f(x), f(y))}{d_Y(f(x), f(z))} &\leq& \frac{2}{3 \lambda(1-\lambda)} \frac{d_Y(f(x), f(v))}{d_Y(f(x), f(z))} 
	\\
	&\leq& \frac{2}{3 \lambda(1-\lambda)} \eta_1 \left(\frac{d_X(x,v)}{d_X(x,z)} \right) 
	\\
	&\leq& \frac{2}{3 \lambda(1-\lambda)} \eta_1 \left(\frac{d_X(x,w) + d_X(w,v)}{d_X(x,z)} \right).
\end{eqnarray*}
	
	Thus
	\[
	\frac{d_Y(f(x), f(y))}{d_Y(f(x), f(z))} \leq \frac{1}{3 \lambda(1-\lambda)} \eta_1 \left(\max\{1, 1/\eta^{-1}(1)\} \cdot \frac{d_X(x,y)}{d_X(x,z)} \right).
	\]
\end{proof}

\begin{Remark}
	Since a quasisymmetry can be extended to the completion of its domain(Proposition $10.11$ in \cite{He01}), our restriction of $X_1$ and $X_2$ to be closed can be weakened.
\end{Remark}

\begin{Remark}
	If $X_1$ and $X_2$ forms an open cover of a compact space $X$, Theorem $2.23$ in \cite{TV80} has completely solved this situation.
\end{Remark}

\begin{Remark}
	Lemma \ref{gqs} is also valid in more generality. For example, $X$ and $Y$ are not length spaces but quasi-geodesic spaces(i.e., every two point is connected by a path whose length is comparable to the distance of these two points), or $X$ is covered by subsets of any finite number satisfying some extra conditions in addition to Lemma \ref{gqs}.
\end{Remark}

\subsection{Gromov-Hausdorff distance and weak tangents of metric spaces}

Recall that the \emph{Hausdorff distance} between two subsets $X$ and $Y$ in the same ambient space, denoted by $d_H(X,Y)$, is defined by
\begin{equation}
d_H(X,Y) : = \inf\{r>0: X \subset N_r(Y) 
\ \textrm{and} \ Y \subset N_r(X)\}.
\end{equation}

The \emph{Gromov-Hausdorff distance} $d_{GH}(X,Y)$ of two metric spaces $X$ and $Y$ (they do not need to be in the same space), is defined by
\begin{equation}\label{GH}
d_{GH}(X,Y) : = \inf_{f,g}\{ d_H(f(X),g(Y))\}.
\end{equation}
where $f:X \to Z$ and $g:Y \to Z$ are isometric embeddings into some metric space $Z$.

A sequence of compact metric spaces $\{X_n\}_{n=1}^\infty$ converges in the \emph{Gromov-Hausdorff sense} to a compact metric space $X$ if $d_{GH}(X_n,X) \to 0$ as $n \to \infty$. We denote it by $X_n \xrightarrow{GH} X$. Notice that the limit space is unique up to isometries. If $X_n$ and $X$ are in the same ambient space, then $d_H(X_n, X) \to 0$ implies $X_n \xrightarrow{GH} X$.

A map $f: X \to Y$ is called an \emph{$\eps$-isometry} if
\begin{equation*}
\dist(f) : = \sup\{ |d_X(x_1,x_2) - d_Y(f(x_1),f(x_2))| :x_1, x_2 \in X \ \} \leq \eps
\end{equation*}
and
\begin{equation*}
d_H(f(X),Y) \leq \eps.
\end{equation*}

The following propositions play an important role in Gromov-Hausdorff convergence. See Section $7.3$ and $7.4$ of \cite{BBI01} for a reference.

\begin{Proposition}\label{epsisometry}
	Let $X$ and $Y$ be two metric spaces and $\epsilon >0$.
	\begin{enumerate}
		\item If $d_{GH}(X,Y) < \eps$, then there exists a $2\eps$-isometry from $X$ to $Y$.
		\item If there exists an $\eps$-isometry from $X$ to $Y$, then $d_{GH}(X,Y) < 2\eps$.
	\end{enumerate}
\end{Proposition}

\begin{Proposition}[Precompactness]\label{precompactness}
	Let $\mathfrak{X}$ be a collection of compact, uniformly bounded and doubling metric spaces. If the doubling constant of every element in $\mathfrak{X}$ is uniformly bounded, then any sequence of elements of $\mathfrak{X}$ contains a convergent subsequence in the Gromov-Hausdorff sense.
\end{Proposition}

A \emph{pointed metric space} is a triple $(X, p, d_X)$ where  $(X, d_X)$ is a metric space with a base point $p \in (X, d_X)$. The \emph{pointed Gromov-Hausdorff convergence} is an analog of Gromov-Hausdorff convergence appropriate for non-compact spaces. A sequence of pointed metric spaces $\{(X_n, p_n, d_n)\}$ converges in the \emph{pointed Gromov-Hausdorff sense} to a complete metric space $(X, p, d_X)$ if for every $r > 0$ and $\eps > 0$ there exists a $n_0 \in \mathbb{N}$ such that for every $n > n_0$ there exists a map $f : B(p_n, r) \to X$ such that the following hold:
\begin{enumerate}
	\item $f(p_n) = p$;
	\item $\dist(f) < \eps$;
	\item $B(p, r-\eps) \subset N_\eps(f(B(p_n,r)))$.
\end{enumerate}
We also denote this type of convergence by $(X_n, p_n, d_n) \xrightarrow{GH} (X, p, d_X)$. Readers can verify that there is no difference between open and closed balls in the definition. Intuitively, the ball $B(p_n,r)$ in $X_n$ lies within the Gromov-Hausdorff distance of order $\epsilon$ from a subset of $X$ between the ball of radii $r-\epsilon$ and $r+\epsilon$ centered at $p$. We call a map which is an $\epsilon$-isometry and keeps the base point a \emph{pointed $\epsilon$-isometry}.

\begin{figure}[htbp]
	\centering
	\includegraphics[width=5in]{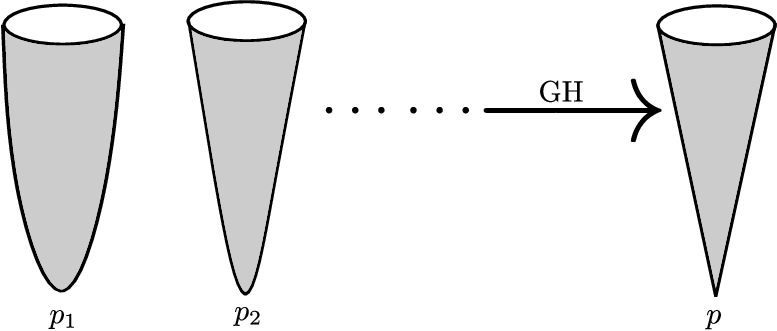}
	\caption{An illustration of pointed Gromov Hausdorff convergence.}
\end{figure}

The definition of pointed Gromov-Hausdorff convergence varies in different sources. We choose the one which balances intuition and generality. If $X$ is a length space, then our definition of Gromov-Hausdorff convergence implies the following one: $\overline{B}(p_n, r) \xrightarrow{GH} \overline{B}(p,r)$ for every $r >0$. Notice that we have abused notation slightly here. Although $\overline{B}(p_n, r)$ may not be compact; however, we can still show that $d_{GH}(\overline{B}(p_n, r), \overline{B}(p,r)) \to 0$. The next result is about the pointed Gromov Hausdorff limit of length spaces. Readers may see Section $8.1$ of \cite{BBI01} for a reference.

\begin{Theorem}\label{lspgh}
	Let $(X_n, p_n) \xrightarrow{GH} (X, p)$ where $X_n$ are length spaces and $X$ is complete. Then $X$ is a length space.
\end{Theorem}

Readers can find more information in Chapter $7$ and $8$ of \cite{BBI01}.

We denote by $f: (X,p, d_X) \to (Y,q,d_Y)$ a map between two pointed metric spaces which fixes the base points. Namely,  $f$ is a map between $(X, d_X)$ and $(Y, d_Y)$ and $f(p) = q$.

Let ${(X_n, p_n, d_n)}$ and ${(Y_n, q_n, l_n)}$ be sequences of pointed metric spaces that converge in the pointed Gromov-Hausdorff sense to $(X, p, d)$ and $(Y, q, l)$, respectively. Let ${f_n}$ be a sequence of mappings given by $f_n: (X_n, p_n, d_n) \to (Y_n, q_n, l_n)$ for each $n \in \mathbb{N}$ and let $f:(X, p, d) \to (Y, q, l)$. We say \emph{$f_n$ converges to $f$} if there exist pointed $\epsilon_n$-isometries
\[
\varphi^n : (X, p) \to (X_n, p_n) \qquad \textrm{and} \qquad \psi_n: (Y_n, q_n) \to (Y, q)
\]
such that the following holds:
\[
\lim_{n \to \infty} \psi_n \circ f_n \circ \varphi^n(x) = f(x)
\] 
for any $x \in X$ and $\epsilon_n \to 0$ as $n \to  \infty$. Naturally, we have $f(p) =q$.

Let $(X, d_X)$ be a metric space and $p$ be a point in $X$. We call a pointed metric space a \emph{weak tangent} of $X$ at $p$ if it is the pointed Gromov-Hausdorff limit of a sequence of spaces $\{(X, p_n, d_X/\lambda_n)\}$ where $\lambda_n \to 0$ and $p_n \to p$ in $X$. We denote by $(T_pX, p_\infty, g_p)$ the above weak tangent. A specific case of the weak tangents is when $p_n = p$ for all $n$, and we call such a weak tangent a \emph{proper weak tangent} at $p$.

We denote by $WT_p(X)$ and $WT(X)$ the collections of all weak tangents at $p$ of $X$ and all weak tangents of $X$, respectively. Similarly, we denote by $PWT_p(X)$ and $PWT(X)$ the collections of all proper weak tangents at $p$ of $X$ and all proper weak tangents of $X$, respectively.

A natural question is whether there exists at least one weak tangent at every point. The answer is given in the following proposition, which is Theorem $8.1.10$ in \cite{BBI01}.

\begin{Proposition}\label{psprecompactness}
	Let $\mathfrak{X}$ be a collection of pointed metric spaces. Suppose that for every $r > 0$ there exists a constant $C$ depending on $r$ such that for every $(X,p) \in \mathfrak{X}$, the ball $B(p,r)$ is $C$-doubling. Then any sequence of elements of $\mathfrak{X}$ contains a convergent subsequence in the pointed Gromov-Hausdorff sense. 
\end{Proposition}

Proposition \ref{psprecompactness} shows that there exist at least one weak tangent at every point of a doubling metric space. Proposition \ref{psprecompactness} is analogous to Proposition \ref{precompactness}. Roughly speaking, it requires the property of uniformly local doubling to induce precompactness.

\section{Quasisymmetric embeddability is hereditary}\label{QEIH}

In this section we first prove Theorem \ref{QSweakT}, then establish the following result: If $X$ is quasisymmetrically embedded into a normed vector space, then every element in $WT(X)$ is quasisymmetrically embedded into the same space, given that $X$ is proper and doubling. 

The proof of Theorem \ref{QSweakT} is mainly finished by next Lemma. See Lemma 2.4.7 in \cite{KL04}.

\begin{Lemma}[Keith, Laakso]\label{upfc}
	Let ${(X_n, p_n, d_n)}$ and ${(Y_n, q_n, l_n)}$ be sequences of proper pointed metric spaces that converge in the pointed Gromov-Hausdorff sense to $(X, p, d)$ and $(Y, q, l)$, respectively. Let $f_n : (X_n, p_n) \to (Y_n, q_n)$ be an $\eta$-quasisymmetric map for each $n \in \mathbb{N}$, where $\eta$ is fixed. If there exist a $C \geq 1$ and a sequence $\{x_n\}$, where each $x_n \in X_n$, such that
	\begin{equation}\label{dib}
	\frac{1}{C} \leq d_n(p_n, x_n) \leq C \qquad \textrm{and} \qquad \frac{1}{C} \leq l_n(q_n, f_n(x_n)) \leq C
	\end{equation}
	for every $n \in \mathbb{N}$. Then, after passing to a subsequence, we have $\{f_n\}$ converges to some $\eta$-quasisymmetric map $f$ between $X$ and $Y$ and $f(p) = q$.
\end{Lemma}

The proof is based on the equicontinuity of quasisymmetric maps and the Arzel\'a-Ascoli theorem. For the sake of completeness, we will illustrate a concise proof of the above lemma. Notice that our definitions of pointed Gromov-Hausdorff convergence and weak tangent are not the same as that in \cite{KL04}; however, the idea is still the same. We first introduce two definitions.

A family of mappings $f_n: (X_n, p_n, d_n) \to (Y_n, q_n, l_n)$ are \emph{equicontinuous on bounded subsets} if for any $R>0$, and for any $\eps>0$, there exists a $\delta>0$ such that
\[
l_n(f_n(x),f_n(x')) < \eps \ \textrm{whenever} \ x,x' \in B(p_n,R) \ \textrm{and} \ d_n(x,x') < \delta.
\]
A family of mappings $f_n: (X_n, p_n, d_n) \to (Y_n, q_n, l_n)$ are \emph{uniformly bounded on bounded sets} if for any $R > 0$,
\[
\sup_{n}\sup_{x,x' \in B(p_n,R)}l_n(f_n(x),f_n(x')) < \infty.
\]

\begin{proof}[Proof of Lemma \ref{upfc}]
	$\{f_n\}$ are equicontinuous on bounded subsets due to inequalities \eqref{dib} in Lemma \ref{upfc}, and the proof is analogous to Proposition $10.26$ in \cite{He01}. Similarly, inequalities \eqref{dib} in Lemma \ref{upfc} imply that $\{f_n\}$ are uniformly bounded on bounded sets.
	
	We first assume that $X$ and $Y$ are closed and bounded; thus it can be directly verified that $X_n$ and $Y_n$ should be uniformly bounded. Let's consider the following diagram:
	\[
	\xymatrix{
	X_n \ar[rr]^{f_n} && Y_n \ar[dd]^{\psi_n}
	\\
	\\
	X \ar[rr]_f \ar[uu]^{\varphi^n} && Y.
	}
	\]
	In this diagram, $\varphi^n$ and $\psi_n$ are pointed $\frac{1}{n}$-isometries from $X$ to $X_n$ and $Y_n$ to $Y$, respectively. We will construct a function $f$ such that the above diagram commutes when $n \to \infty$.
	
	Since $X$ is proper, it is compact and separable; thus we can select a dense countable subset $E$ of $X$. For any $x_1 \in E$, the sequence $\psi_n \circ f_n \circ \varphi^n(x_1)$ is bounded since $f_n$ is uniformly bounded on bounded sets. $Y$ is a proper space; thus there exists a subsequence $n_i$ such that $\psi_{n_i} \circ f_{n_i} \circ \varphi^{n_i}(x_1)$ converges to a point $y_1$ in $Y$. We define $f(x_1) := y_1$.  By Cantor's diagonal argument, there exists a subsequence $\{n_j\}$ such that we are able to define a function $f$ on $E$ where
	\[
	f(x) := \lim_{j \to \infty} \psi_{n_j} \circ f_{n_j} \circ \varphi^{n_j}(x)
	\]
	for any $x \in E$.  Combining this with the equicontinuity on bounded subsets, we can follow a proof analogous to the Arzel\'a-Ascoli theorem to define a function $f: (X, p, d) \to (Y, q, l)$ as the limit of $\{f_{n_j}\}$. Furthermore, $f$ is a homeomorphism.
	
	It requires some basic tricks of pointed $\epsilon$-isometries(see Proposition \ref{epsisometry}) to show that $f$ is a quasisymmetry. We still denote the subsequence $\{n_j\}$ as $\{n\}$ for convenience here. Roughly speaking, for any $x,y,z \in X$, there exists a $1/n$-isometry $\varphi_n : X \to X_n$ where the distance between any two points of $x,y,z$ is much larger than $1/n$. Similarly, we can also choose a $\psi_{n'}$ such that the distance between any two points of $f_n \circ \varphi^n(x), f_n \circ \varphi^n(y), f_n \circ \varphi^n(z)$ is much larger than $1/n'$. Finally, passing $x,y,z$ through $\psi_{n'} \circ f_n \circ \varphi^n$ and letting $n,n' \to \infty$ will show the three point condition of quasisymmetries. We leave the details to the readers.
	
	If $X$ and $Y$ are unbounded, we select bounded exhaustions $\{U_n\}$ and $\{V_n\}$ of $X$ and $Y$(namely, a sequence of bounded subsets of the ambient space whose union equals to the whole space), respectively, such that each $U_n$ is $\eta$-quasisymmetric equivalent to $V_n$. Applying Cantor's diagonal argument again, we finish the proof by taking a sub-limit.
\end{proof}
 
\begin{proof}[Proof of Theorem \ref{QSweakT}]
	
	Since $X$ is doubling, we assume that $\{(X, p_n, d_X/\lambda_n)\}$ converges to $T_pX$ and $q_n : = f(p_n)$. We split the proof into two situations: either $X$ is uniformly perfect or not.
	
	If $X$ is uniformly perfect, then there exists a constant $C \geq 1$ and a sequence of $x_n \in X$ such that
	\begin{equation}\label{up}
	\frac{\lambda_n}{C} \leq d_X(p_n, x_n) \leq \lambda_n.
	\end{equation}
	Let $\omega_n =d_Y(q_n, f(x_n))$. Thus we define $f_n:(X, p_n, d_X/\lambda_n) \to (Y, q_n, d_Y/\omega_n)$ by $f_n(x) := f(x)$. Notice that $\omega_n \neq 0$, $\omega_n \to 0$ as $n \to \infty$ and $f_n$ are still $\eta$-quasisymmetric.
	
	An application of Lemma \ref{upfc} on $\{f_n\}$ with the following inequalities
	\begin{equation}
	\frac{1}{C} \leq \frac{d_X}{\lambda_n}(p_n, x_n) \leq 1 \qquad \textrm{and} \qquad \frac{d_Y}{\omega_n}(q_n, f_n(x_n)) = 1
	\end{equation}
	finishes the proof.
	
	If $X$ is not uniformly perfect, we split the proof into two cases.
	
	Case $1$. Notice that $\{p_n\}$ and $\{\lambda_n\}$ are fixed. If inequality (\ref{up}) is true for some sequence $\{x_n\}$ and some constant $C \geq 1$, then the above argument finishes the proof.
	
	Case $2$. If not, every sequence of $x_n \in X$ should satisfy
	\begin{enumerate}
		\item $d_X(p_n, x_n)/\lambda_n \to 0$ when $d_X(p_n,x_n)/\lambda_n$ is uniformly bounded from above.
		
		\item $d_X(p_n, x_n)/\lambda_n \to \infty$ when $d_X(p_n,x_n)/\lambda_n$ is uniformly bounded from below.
	\end{enumerate}
	
	For any sequence $\{a_n\}$ where each $a_n \in B(p_n, \lambda_n)$, we have that $d_X(p_n, a_n)/\lambda_n < 1$ for every $n$. Thus $d_X(p_n, a_n)/\lambda_n \to 0$. For any sequence $\{b_n\}$ where each $b_n \in X \setminus \overline{B}(p_n, \lambda_n)$, we have that $d_X(p_n, b_n)/\lambda_n > 1$ for every $n$. Thus $d_X(p_n, b_n)/\lambda_n \to \infty$. $\partial B(p_n, \lambda_n)$ should be empty when $n$ is sufficiently large, otherwise inequality (\ref{up}) is true for some $\{x_n\}$ and some $C \geq 1$.
	
	Thus $(X, p_n, d_X/\lambda_n)$ is inside the complement of a metric annulus $B(p_n, R_n) \setminus \overline{B}(p_n, r_n)$ where $\lim\limits_{n \to \infty}R_n = \infty$ and $\lim\limits_{n \to \infty}r_n = 0$. It means that $(X, p_n, d_X/\lambda_n) \xrightarrow{GH} \{p\}$. 
	
	If we restrict $X$ and $Y$ to bounded subsets that contain of $p_n$ and $q_n$, respectively. It will not affect the weak tangents and a chosen of $\omega_n$ by $\lambda_n$ and Proposition 2.1 will lead to $(Y, q_n, d_Y/\omega_n) \xrightarrow{GH} \{q\}$.	 
\end{proof}

\begin{Remark}
	Lemma \ref{upfc} and Theorem \ref{QSweakT} are still valid in some other situations. Namely, they are also valid for maps that are isometric ($l_n(f_n(x),f_n(y)) = d_n(x,y)$), $r$-similar ($l_n(f_n(x),f_n(y)) = r \cdot d_n(x,y)$), and $L$-bi-Lipschitz ($1/L \cdot d_n(x,y) \leq l_n(f_n(x),f_n(y)) \leq L \cdot d_n(x,y)$). The proofs are verbatim the same.
\end{Remark}

We state the following generalized lemma without proof which will be repeatedly used in later contents:

\begin{Lemma}\label{upfci}
	Let ${(X_n, p_n, d_n)}$ and ${(Y_n, q_n, l_n)}$ be sequences of proper pointed metric spaces that converge in the pointed Gromov-Hausdorff sense to $(X, p, d)$ and $(Y, q, l)$, respectively. Let $f_n : (X_n, p_n) \to (Y_n, q_n)$ be an isometric/$r$-similar/$L$-bi-Lipschitz/$\eta$-quasisymmetric map for each $n \in \mathbb{N}$. Suppose there exist a $C \geq 1$ and a sequence $\{x_n\}$, where each $x_n \in X_n$, such that
	\begin{equation}
	\frac{1}{C} \leq d_n(p_n, x_n) \leq C \qquad \textrm{and} \qquad \frac{1}{C} \leq l_n(q_n, f(x_n)) \leq C
	\end{equation}
	for every $n \in \mathbb{N}$. Then, after passing to a subsequence, $\{f_n\}$ converges to some isometric/$r$-similar/$L$-bi-Lipschitz/$\eta$-quasisymmetric map $f$ between $X$ and $Y$ and $f(p) = q$.
\end{Lemma}

We are interested in a metric space which quasisymmetrically admits all weak tangents of its subsets, i.e., a metric space $X$ with the property that every weak tangent of every $Y \subset X$ can be uniformly quasisymmetrically embedded into $X$. We call such a space \emph{quasisymmetric-tangent-self-embeddable}. The following lemma proves that a finite dimensional normed vector space admits this property.

\begin{Lemma}\label{qsses}
	Let $Y$ be a subset of a finite dimensional normed vector space $(X, d)$. Then every weak tangent of $Y$ is (isometrically) a subset of $X$.
\end{Lemma}

\begin{proof}
	Since $X$ is finite dimensional, it is doubling and proper. In fact, doubling and proper are equivalent to finite dimension here. Let $T_pY$ be a weak tangent of $Y$ and $(Y, p_n, d/\lambda_n) \xrightarrow{GH} (T_pY, p_\infty, g_p)$. Let 
	\[
	\lambda_nY = \{x \in X : \textrm{there exists a} \ y \in Y \ \textrm{such that} \ x-p_n = \lambda_n(y-p_n) \}
	\]
	be a linear dilation of $Y$ at the point $p_n$. Thus $(\lambda_nY, p_n, d)$ is isometric to $(Y, p, d/\lambda_n)$ for any $\lambda_n$.
	
	We claim that there exists a subsequence $\{\lambda_{n_i}\}$ and a $Y_p \subset X$ such that for every $r > 0$, $\overline{B}(p_{n_i},r) \cap \lambda_{n_i} Y$ converges in the Hausdorff sense to $\overline{B}(p,r) \cap Y_p$. The proof is analogous to Proposition \ref{psprecompactness}. This is achieved by Theorem $7.3.8$ in \cite{BBI01} and Cantor's diagonal argument. Thus $(\lambda_{n_i}Y, p_{n_i}, d) \xrightarrow{GH} (Y_p, p, d)$. Since $(Y, p, d/\lambda_{n_i}) \xrightarrow{GH} (T_pY, p_\infty, g_p)$, we finish the proof by applying	 the isometric version of Lemma \ref{upfci}.
\end{proof}

The following corollary follows Theorem \ref{QSweakT} and Lemma \ref{qsses}.

\begin{Corollary}\label{nvs}
	Let $X$ be a proper, doubling metric space and $V$ be a finite dimensional normed vector space. If $X$ can be $\eta$-quasisymmetrically embedded into $V$, then any weak tangent of any subset of $X$ can be $\eta$-quasisymmetrically embedded into $V$.
\end{Corollary}

Corollary \ref{nvs} is still valid if we substitute the normed vector space with a quasisymmetric-tangent-self-embeddable space(with a little revision). However, we know very few examples of these spaces. Finite dimensional normed vector spaces are one example, and it is plausible that any quasi-self-symmetric space(see section \ref{LEIGE}) with a suitable ``extension" to infinity will be another.

\section{Counterexamples}\label{CE}

In this section, we will illustrate several counterexamples to show that the converse of Theorem \ref{QSweakT} is not true.

\subsection{Slit Sierpi\'nski carpets}
Let $R=(a_1,b_1) \times (a_2,b_2) \subset \mathbb{R}^2$ be a rectangle and let $s=\{x\} \times [c_2,d_2]\subset R$ be a \emph{slit} in $R$, where $l(s)=\mbox{diam}(s)=d_2-c_2$ is the \emph{length} of $s$. $R\backslash \bigcup_{i=1}^m s_i$, where $s_i$'s are mutually disjoint slits in $R$, is called a \emph{slit domain}.

We say that $\Delta\subset [0, 1]^2$ is a \emph{dyadic square of generation $n$} if there exist $i,j \in \{0,1,2,\ldots,2^n-1\}$ such that
\begin{align*}
\Delta=\left[\frac{i}{2^n},\frac{i+1}{2^n}\right] \times \left[\frac{j}{2^n},\frac{j+1}{2^n}\right].
\end{align*}
The sidelength of a dyadic square $\Delta$ will be denote by $l(\Delta)$.

Given an infinite sequence $\mathbf{r}=\{r_i\}_{i=0}^\infty$ such that $r_i\in(0,1)$, we define a slit domain $$\mathcal{S}_n(\mathbf{r}): =[0,1]^2\backslash \left( \bigcup_{i=0}^n\bigcup_{j=1}^{2^i} s_{ij}\right),$$
where
\begin{enumerate}
	\item $s_{ij}\subset \Delta_{ij}$, where $\Delta_{ij}$ is a dyadic square of generation $i$.
	\item The center of $s_{ij}$ coincides with the center of $\Delta_{ij}$.
	\item $l(s_{ij_1})=l(s_{ij_2})=r_i\cdot\frac{1}{2^i}$ for $j_1,j_2\in\{1,\ldots,2^i\}$.
\end{enumerate}

\begin{figure}[htbp!]
	\includegraphics[height=1.8in]{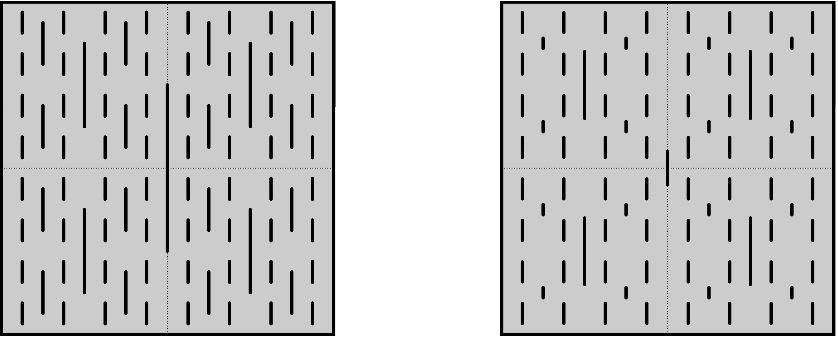}
	\caption{Slit domains with respect to $(\frac{1}{2},\frac{1}{2},\frac{1}{2},\frac{1}{2})$ and $(\frac{1}{10},\frac{2}{5},\frac{1}{8},\frac{1}{2})$.}
\end{figure}

We equip $\mathcal{S}_n(\mathbf{r})$ with the intrinsic metric and denote by $\mathscr{S}_n(\mathbf{r})$ the completion of $\mathcal{S}_n(\mathbf{r})$. One should be aware that the boundary of every slit is a topological circle after the completion.

For every $m, n \in \mathbb{N} \cup \{0\}$ with $m \leq n$ there exists a natural $1$-Lipschitz projection $$\pi_{m,n}: \mathscr{S}_n \to \mathscr{S}_m$$ obtained by identifying the points on the slits of $\mathscr{S}_n$ that correspond to the same point of $\mathscr{S}_m$.

As a topological space, the \emph{dyadic slit Sierpi\'nski carpet corresponding to $\mathbf{r}$} is defined as the inverse limit of the system $(\mathscr{S}_n,\pi_{m,n})$, and is denoted by $\mathscr{S}_{\mathbf{r}}$. More explicitly,
\begin{align}
\mathscr{S}_{\mathbf{r}} = \left\{ (p_0,p_1,\ldots) : p_i \in \mathscr{S}_i \mbox{ and } p_i=\pi_{i,i+1}(p_{i+1}) \right\}.
\end{align}
In the following content, we just use the notation $\S$ when there is no ambiguity.

The diameter of each $\mathscr{S}_n$ is clearly bounded by $2$. If $x = (x_0,x_1, \ldots )$ and $
y = (y_0, y_1, \ldots)$ are points in $\mathscr{S}$, we define a distance between them by
\[
d_{\mathscr{S}}(x,y)=\lim_{n\to\infty} d_{\mathscr{S}_n}(x_n,y_n)
\]
Since every $\pi_{m,n}$ is $1$-Lipschitz, $(d_{\mathscr{S}_n}(p_n, q_n))$ is a non-decreasing bounded sequence, and thus $d_{\mathscr{S}}(p, q)$ exists and defines a metric on $\mathscr{S}$. It can be directly verified that $d_\S$ is a length metric on $\S$.

 $(\mathscr{S}, d_{\S})$ is a metric Sierpi\'nski carpet, i.e., a metric space which is homeomorphic to the standard Sierpi\'nski carpet. When we mention a \emph{slit} in $\mathscr{S}$, we mean the boundary component of $\mathscr{S}$ which corresponds to a slit in the slit domain. For more information, readers may see \cite{HL} for a reference.

The dyadic slit carpets were first studied in \cite{Me10}, in which Merenkov investigated the carpet with respect to $\{r_i\}_{i = 0}^\infty$ where $r_i = 1/2$ for any $i$. 

The following theorem from \cite{HL} classifies the planar quasisymmetric embeddability of dyadic slit carpets.

\begin{Theorem}[Hakobyan, Li]\label{planarem}
	There exists a quasisymmetric embedding of $\mathscr{S}({\mathbf{r}})$ into $\mathbb{R}^2$ if and only if $\mathbf{r}=\{r_i\}_{i=0}^{\infty} \in \ell^2$.
\end{Theorem}

We call a dyadic slit carpet $\S(\mathbf{r})$ \emph{harmonic} if $\{r_i\} \notin \ell^2$ and $r_i \to 0$ as $i \to \infty$. For example, a dyadic slit carpet with respect to $\mathbf{r} = \{\frac{1}{\sqrt{i}}\}$ is harmonic. Theorem \ref{planarem} implies that there exists no quasisymmetric embedding from a harmonic slit carpet into $\mathbb{R}^2$.

\begin{Proposition}\label{scwte}
	Any weak tangent of a harmonic $\S$ is either the closed one quarter plane  $Q$, the closed half plane $H$, $\mathbb{R}^2$ or $T$, where $T$ is the completion of $\mathbb{R}^2 \setminus (0, \infty)$ equipped with the intrinsic metric.
\end{Proposition}

\begin{proof}
	The proof of Proposition \ref{scwte} is rather intuitive. A metric carpet $S$ is \emph{porous} at a point $p$ if there exists a constant $C > 1$ depending on $p$ such that for any $0 < r < \diam(X)$, there exists a boundary component $\Lambda$ of $S$ in $B(p,r)$ such that
	\[
	\frac{1}{C} < \frac{\diam(\Lambda)}{r} < C.
	\]
	
	Our proof is based on the idea that if $\S$ is harmonic, then $\S$ is not porous at any point. In simple words, for any $p \in \S$ and any $0 < r < \diam(X)$, the diameter of any slit inside $B(p,r)$ will decrease to $0$ as ``blowing up" the carpet near $p$; thus ``erase" all the slits inside $B(p_\infty,r)$. 
	
	We should first be aware that all the weak tangents of $\S$ are length spaces by Theorem \ref{lspgh}. Let $p$ be a point that is not on the boundary of $\mathscr{S}$. We define $\pi$ to be the natural projection from $\mathscr{S}$ to $[0,1]^2$ i.e., $\pi$ project any point on the slit carpet to the corresponding point on the unit square. For any $r>0, \lambda > 0$, let $f_n: =1/\lambda \cdot \pi|_{B(p,\lambda r)}$. Thus it is an embedding into $\mathbb{R}^2$ and satisfies 
	\begin{enumerate}
		\item $f(p) = p$;
		\item $\dist(f) < \eps$ for some $\eps$ depends on the length of the slits in $B(p,r)$;
		\item $B(p, r-\eps) \subset N_\eps(f(B(p,r)))$.
	\end{enumerate}
	Since the slit carpet is harmonic, the length of the slits in $B(p,r)$ and $\eps$ goes to $0$ as $\lambda \to 0$ . Thus by definition any proper weak tangent at $p$ is $\mathbb{R}^2$.
	
	We follow the same idea to construct pointed $\eps$-isometries, then it provides the following cases of proper weak tangents:
	
	\begin{enumerate}
		\item Any proper weak tangent at a corner point of $\S$ is $Q$.
		
		\item Any proper weak tangent at a non-corner point on the outer boundary(largest boundary) of $\S$ is $H$.
		
		\item Any proper weak tangent at an endpoint of a slit is $T$.
		
		\item Any proper weak tangent at a point on a slit which is not the endpoint is $H$. 
		
		\item Any proper weak tangent at a non-boundary point is $\mathbb{R}^2$.
	\end{enumerate}
	
	In next step, we will show that all the weak tangents of $\S$ should be $\{Q,H,\mathbb{R}^2, T\}$. Suppose that there exists $\{p_n\}$ s.t. $p_n \to p$ and $\left(\S, p_n, \frac{d_{\S}}{\lambda_n}\right)$ converges to a weak tangent $T_p\S$. If $n$ is sufficiently large, $\left(\S, p_n, \frac{d_{\S}}{\lambda_n}\right)$ is $\eps$-close to one of $\{Q,H,\mathbb{R}^2, T\}$ in the pointed Gromov-Hausdorff sense where $\eps$ only depends on $\mathscr{S}$ and $\lambda_n$ and $\{Q, H, \mathbb{R}^2, T\}$ are all distinct in this sense. Thus $\left(\S, p_n, \frac{d_{\S}}{\lambda_n}\right)$ should converges to one of them which finishes the proof.
\end{proof}

\begin{Corollary}\label{nonqsc}
	There exists a metric Sierpi\'nski carpet which is not quasisymmetrically embedded into $\mathbb{R}^2$ but every weak tangent is uniformly bi-Lipschitzly embedded into $\mathbb{R}^2$.
\end{Corollary}

\begin{proof}
	Let's take a harmonic dyadic slit carpet and what is left to show is the embeddability of $T$. We define a function $\Phi: T \to H$, which identifies the points by halfing the argument, and by showing it is a bi-Lipschitz map we finish the proof. More precisely, we define
	\begin{eqnarray*}
	\Phi : T & \longrightarrow & H
	\\
	r e^{i\theta} & \mapsto & r e^{i\frac{\theta}{2}}
	\end{eqnarray*}
	for $r \geq 0$ and $\theta \in (0, \pi)$. However, we do require that $\Phi(r e^{i\pi}) \neq \Phi(r e^{i0})$ for $r \neq 0$. Notice that $\partial T$ is a real line, thus we can define $\partial T = R_1 \cup R_2$ where both $R_1$ and $R_2$ represent $[0, +\infty)$ under the natural projection of $T$ to $\mathbb{R}^2$. We extend $\Phi$ to its boundary in a continuous and symmetric way.
	
	Let $(r,\theta)$ be a point that is represented in polar coordinate in $T$, then the Jacobian matrix of $\Phi$ at $(r,\theta)$ is
	\[
	\left(\begin{array}{ll}
	1 & 0\\
	0 & \frac{1}{2}
	\end{array}
	\right).
	\]
	It is clear that $\Phi$ is locally $L$-bi-Lipschitz at every point in $T \setminus \partial T$ for some $L \geq 1$. We will prove that $\Phi$ is bi-Lipschitz on $T \setminus \partial T$ in next step. 
	
	Let $x,y \in T \setminus \partial T$ and $\gamma$ be the geodesic connecting $x$ and $y$. Let $x_1, \ldots, x_{n+1}$ be points on $\gamma$ such that $\Phi$ is bi-Lipschitz on $\{x_{i}, x_{i+1}\}$ for any $1 \leq i \leq n$. For convenience, we take $x = x_1, y =x_{n+1}$. Thus
	\[
	|\Phi(x)-\Phi(y)| \leq \sum_{i=1}^n|\Phi(x_{i+1})-\Phi(x_i)| \leq \sum_{i=1}^{n}Ld_T(x_i,x_{i+1})  = Ld_T(x,y)
 	\]
	Similarly, we also have $d_T(x,y) \leq L |\Phi(x)-\Phi(y)|$.
	
	Extending $x,y$ to $R_1 \cup R_2$ with continuity finishes the proof.
\end{proof}

Let $\mathbf{S}=\{s_i\}_{i=0}^\infty$ be the collection of all slits in $\mathscr{S}(\mathbf{r})$. We denote by $\mathscr{P}_i$ the metric space generated by identically gluing three sides of two identical squares with sidelength $l(s_i)$. In this way, $\mathscr{P}_i$ looks like a ``square pillow" with an ``open mouth". We define $\widehat{\mathscr{S}}(\mathbf{r})$ as

\begin{equation}
\widehat{\mathscr{S}}(\mathbf{r}) := \mathscr{S}(\mathbf{r}) \bigsqcup_{\mathbf{G}} \{\mathscr{P}_i\}_{i=0}^\infty,
\end{equation}
where $\mathbf{G} = \{g_i\}$ is a sequence of gluing functions. Each $g_i$ in $\mathbf{G}$ identically glues the slit $s_i$ with $\partial \mathscr{P}_i$, the topological boundary of $\mathscr{P}_i$.

\begin{figure}[htbp]
	\centering
	\includegraphics[width=5in]{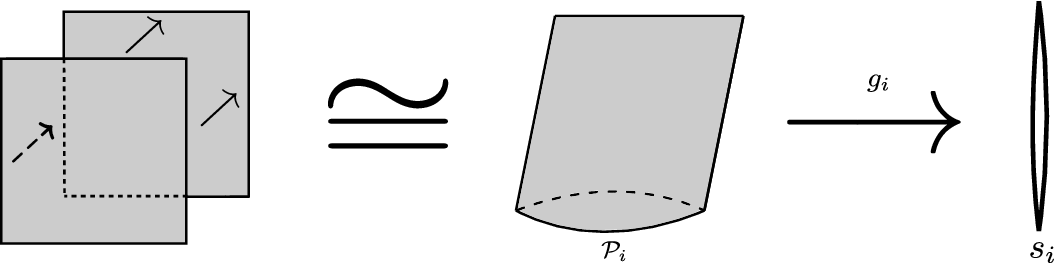}
	\caption{``Square pillow" $\mathscr{P}_i$, gluing function $g_i$ and slit $s_i$.}
	\label{squarepillow}
\end{figure}

$\widehat{\mathscr{S}}$ is homeomorphic to $\mathbb{D}$ and we equip it with intrinsic metric $d_{\widehat{\mathscr{S}}}$. In the original study of $\widehat{\mathscr{S}}$ in \cite{HL}, a more technical metric is equipped on $\widehat{\mathscr{S}}$; however it appears that these two metrics are bi-Lipschitz equivalent, which will not affect our following results. See details of this part in Section $8$ of \cite{HL}

Similarly, we call $\widehat{\mathscr{S}}$ \emph{harmonic} if the corresponding $\S$ is harmonic. Notice that $\S$ can be bi-Lipschitzly embedded into $\widehat{\mathscr{S}}$ by Section $8$ of \cite{HL}, thus any harmonic $\widehat{\mathscr{S}}$ can not be quasisymmetrically embedded into $\mathbb{R}^2$, otherwise $\mathscr{S}$ can be quasisymmetrically embedded into $\mathbb{R}^2$.

To study the weak tangents of $\widehat{\mathscr{S}}$, we introduce several new spaces. Recall that $Q  = \{(x,y) \in \mathbb{R}^2: x \geq 0, y\geq 0 \}$ and we define
\[
D= Q \bigsqcup_{d} Q.
\]
where $D$ is a length space and $d$ glues $\partial Q$ with $\partial Q$ identically.  Similarly,
\[
L = T \bigsqcup_{l} H 
\]
where $l$ glues $\partial T$ with $\partial H$ isometrically. Notice that both $\partial H$ and $\partial T$ are $\mathbb{R}$, so we can glue their boundaries by isometry. We also equip $L$ with the intrinsic metric.

\begin{Proposition}\label{eswte}
	Any weak tangent of a harmonic $\widehat{\mathscr{S}}$ is bi-Lipschitz equivalent to one of the following spaces: $Q$, $H$, $\mathbb{R}^2$, $D$ or $L$.
\end{Proposition}

\begin{proof}
		The proof of Proposition \ref{eswte} follows the same idea of the proof of Proposition \ref{scwte}. We just sketch the proof. Following the idea illustrated in Proposition \ref{scwte} to construct pointed $\eps$-isometries at each point, then it provides the following cases of proper weak tangents:
	
	\begin{enumerate}
		\item Any proper weak tangent at a corner point of $\widehat{\mathscr{S}}$ is $Q$.
		
		\item Any proper weak tangent at a non-corner point on the outer boundary of $\widehat{\mathscr{S}}$ is $H$.
		
		\item Any proper weak tangent at an endpoint of a slit is $L$.
		
		\item Any proper weak tangent at a point on a slit which is not the endpoint is $\mathbb{R}^2$. 
		
		\item Any proper weak tangent at a corner point of $\mathscr{P}_i$ which is not on a slit is $D$.
		
		\item Any proper weak tangent at other points of $\mathscr{P}_i$ is $\mathbb{R}^2$.
		
		\item Any proper weak tangent at a non-boundary point which is not on $\mathscr{P}_i$ is $\mathbb{R}^2$.
	\end{enumerate}
	
	We claim that all the weak tangents of $\widehat{\mathscr{S}}$ should be $\{Q,H,\mathbb{R}^2, D,L\}$. This still comes from the exhaustion of the local behavior of every point on $\widehat{\mathscr{S}}$ and $\{Q,H,\mathbb{R}^2, D,L\}$ are all distinct in the pointed Gromov-Hausdorff sense.
\end{proof}

Recall that a \emph{quasi-plane} or a \emph{quasi-sphere} is a metric space which is quasisymmetric to $\mathbb{R}^2$ or $\mathbb{S}^2$, respectively.

\begin{Corollary}\label{nonqsps}
	Let $X$ be a topological plane or a topological sphere. There exist a metric on $X$ such that $X$ is not a quasi-plane or a quasi-sphere, respectively, but every weak tangent of $X$ is quasisymmetric to $\mathbb{R}^2$.
\end{Corollary}

\begin{proof}
	Let $\widehat{\mathscr{S}}$ be harmonic. If we replace every dyadic unit square in $\mathbb{R}^2$ by $\widehat{\mathscr{S}}$ and equip it with the intrinsic metric, we have the desired metric plane. Similarly, if we glue two copies of $\widehat{\mathscr{S}}$ isometrically by their outer boundaries and equip it with the intrinsic metric, we have the desired metric sphere. Thus in order to finish the proof of Corollary \ref{nonqsps}, it remains to show that $L$ and $D$ are quasisymmetric to $\mathbb{R}^2$. Revise the function $\Phi$ in Corollary \ref{nonqsc} and follow the same idea to finish the proof.
\end{proof}

\begin{Remark}
	In Section $5$ of \cite{Wu19}, Wu constructed a metric $\delta$ on $\mathbb{R}$, which is inspired by the ideas of Rickman's rug and the ``flat" snowflake curves, such that
	\[
	X_n = \left( \mathbb{R} \times \mathbb{R}^{n-1}, \sqrt{\delta^2+|  \ \cdot \ |^2}\right)
	\]
	is not quasisymmetric to $\mathbb{R}^n$, but every weak tangent of $X_n$ is isometric to $\mathbb{R}^n$ for $n \geq 2$.
\end{Remark}

\section{Self-quasisymmetricity and quasi-self-symmetricity}\label{LEIGE}

In this section we investigate the converse implication. The goal of this section is to explore the consequences when the weak tangents of a metric space $X$ are uniformly quasisymmetrically embedded into a metric space $Y$?

An appropriate space in this section should illustrate analogous structures between the local and the global scale; thus inspiring the following definition.

Let $X$ be a metric space and $p$ be a point in $X$. $X$ is \emph{$\eta$-self-quasisymmetric} at $p$ if there exists a $r_p > 0$ such that for any $0 < r < r_p$, there exists a subset $U \subset B(p,r)$ such that $U$ is $\eta$-quasisymmetric to $X$. The number $r_p$ is called the \emph{self-quasisymmetry scale} of $p$.

We say $X$ is \emph{homogeneously self-quasisymmetric} at $p$ if there exists a $C_p>0$ such that for any $0 < r < r_p$, there exists a subset $U$ as above satisfying $r/C_p \leq \diam(U) \leq C_p \cdot r$ and $p \in U$. If $X$ is self-quasisymmetric at a set of points, we say it is \emph{uniformly homogeneous} if $X$ is homogeneous at these points and the homogeneous constant $C_p$ and the self-quasisymmetry scale $r_p$ are uniformly bounded.

Notice that at any point of a self-similar space, it is homogeneously self-quasisymmetric. See Section $5$ of \cite{Hu81} for the definition of self-similarity. Our intuition tells us that self-similar spaces should be a target for a converse implication. In fact, we can show it works in more generality.

\begin{Theorem}\label{sqp}
	Let $X$ be a proper, doubling metric space and $p$ be a point in $X$. Assume that there exists a sequence of points $\{p_n\}$ such that $X$ is $\eta$-self-quasisymmetric at $p_n$ for every $n$ and $p_n \to p$ in $X$. Then there exists a weak tangent $T_pX \in WT_p(X)$ such that $X$ is $\eta'$-quasisymmetrically embedded into $T_pX$, where $\eta'(t) = 1/\eta^{-1}\!\!\left(\frac{1}{t}\right)$.
\end{Theorem}

\begin{proof}
	By definition, we can select a sequence of balls $B(p_n, r_n)$ such that $U_n \subset B(p_n,r_n)$, $r_n \to 0$, and there exists a sequence of $\eta$-quasisymmetric maps
	\[
	f_n: (U_n, q_n,\frac{d_X}{\lambda_n}) \to (X, p, d_X).
	\]
	where $q_n := f_n^{-1}(p)$ and $\lambda_n := \diam(U_n)$. It can be verified that $d_X(q_n, p_n) \to 0$, thus $q_n \to p$ in $X$. Notice that $U_n$ may not be open or unique.
	
	Without loss of generality, assume that $\lambda_n > 0$ for all $n \in \mathbb{N}$ and $\lambda_n \to 0$; otherwise, $X$ is a singleton. We may assume that, by Proposition \ref{psprecompactness}, $(X, q_n, \frac{d_X}{\lambda_n}) \xrightarrow{GH} (T_pX, p_\infty , g_p)$ after passing $\{\lambda_n\}$ to a suitable subsequence, where $T_pX$ is a weak tangent at $p$ . Similarly, $(U_n,q_n,\frac{d_X}{\lambda_n}) \xrightarrow{GH} (U_p, p_\infty , g_p)$ after passing to a suitable subsequence again. It can be directly verified by the isometric version of Lemma \ref{upfci} that $U_p$ is a subset of $T_pX$.
	
	It is trivial to see that $(X,p,d_X) \xrightarrow{GH} (\overline{X},p,d_X)$. Here $\overline{X}$ is the completion of $X$ since every pointed Gromov-Hausdorff limit is complete. In the following proof, we will show that $f_n$ subconverges to an $\eta$-quasisymmetry. 
	
	\[
	\xymatrix{
		(U_n, q_n,\frac{d_X}{\lambda_n}) \ar[rr]^{f_n} \ar[dd]_{GH} &\ar@{=>}[dd]& (X, p, d_X) \ar[dd]^{GH}
		\\
		\\
		(U_p, p_\infty , g_p) \ar[rr]_f  && (\overline{X},p,d_X).
	}
	\]
	
	We claim that there exists a sequence of points $x_n \in U_n$ such that Lemma \ref{upfc} holds for $\{f_n\}$. Take a point $x_1$ in $U_1$ such that $d_X(q_1,x_1)/\lambda_1 = C$ for some constant $C > 0$. We denote by $x_n : = f_n^{-1} \circ f_1(x_1)$. Since $f_n^{-1} \circ  f_1: (U_1, q_1, d_X/\lambda_1) \to (U_n, q_n, d_X/\lambda_n)$ is an $\eta_1$-quasisymmetry where $\eta_1$ depends only on $\eta$, by Proposition \ref{diam}, we have
	\begin{equation}
	\frac{d_X({q_n, x_n)}}{\diam(U_n)} \geq \frac{1}{2 \eta_1\left(\frac{\diam(U_1)}{d_X(q_1,x_1)}\right)}  = \frac{1}{2 \eta_1\left(1/C\right)}.
	\end{equation}
	Thus
	\begin{equation}\label{cp}
	\frac{1}{2 \eta_1\left(1/C\right)} \leq \frac{d_X}{\lambda_n}(q_n,x_n) \leq 1 \qquad \textrm{and} \qquad d_X\left(f_n(q_n), f_n(x_n)\right) = d_X\left(p, f_1(x_1)\right).
	\end{equation}
	
	Since $(U_n,q_n,d_X/\lambda_n) \xrightarrow{GH} (U_p, p_\infty , g_p)$ and $(X,p,d_X) \xrightarrow{GH} (\overline{X},p,d_X)$, applying Lemma \ref{upfc} to $\{f_n\}$ induces an $\eta$-quasisymmetry between $U_p$ and $\overline{X}$. This finishes the proof.
\end{proof}

Theorem \ref{sqe} is a direct corollary of Theorem \ref{sqp}. We can further strengthen Theorem \ref{sqp} if we add more conditions.

\begin{Proposition}\label{lsqp}
	Let $X$ be a proper, doubling metric space and $p$ be a point in $X$. Assume that there exists an open neighborhood $V$ of $p$ such that $X$ is uniformly homogeneously $\eta$-self-quasisymmetric at a dense subset of $V$, then $X$ is $\eta'$-quasisymmetrically embedded into every weak tangent in $WT_p(X)$, where $\eta'(t) = 1/\eta^{-1}\!\!\left(\frac{1}{t}\right)$.
\end{Proposition}

\begin{proof}
	Let $(X, p_n, d_X/\lambda_n) \xrightarrow{GH} (T_pX, p_\infty , g_p)$. Without loss of generality, we assume that $\lambda_n$ is smaller than the self-quasisymmetry scale and $B(p_n, \lambda_n) \subset V$ for all $n \in \mathbb{N}$. There exists a sequence of subsets $\{U_n\}$ as in the proof of Theorem \ref{sqp} such that $U_n \subset B(p'_n, \lambda_n/2) \subset B(p_n, \lambda_n)$, $\diam(U_n)$ is comparable to $\lambda_n$ and $p'_n \in U_n$. Here $p'_n$ is a point at which $X$ is $\eta$-self-quasisymmetric such that $p'_n$ is sufficiently close to $p_n$. 
	
	Without loss of generality, we assume that $p_n \in U_n$; otherwise, we can construct a $U_n$ in the following way: Take a sequence of points $\{p'_{n_j}\}$ at which $X$ is $\eta$-self-quasisymmetric such that $\{p'_{n_j}\}$ converges to $p_n$. Then the corresponding $U_{n_j}$  converges in the pointed Hausdorff sense to a subset $U_n$ after passing to a suitable subsequence. Thus $U_n$ is $\eta$-quasisymmetric to $\overline{X}$.  Since  $p'_{n_j} \in U_{n_j}$ and $U_n$ is closed, we have $p_n \in U_n$. See Lemma \ref{qsses} for a similar proof.
	
	Following the above construction, we may assume that for every $n \in \mathbb{N}$, there exists an $\eta$-quasisymmetry $f_n : (U_n, d_X/\lambda_n) \to (\overline{X}, d_X)$. Notice that
	\[
	(U_n,p_n,\frac{d_X}{\lambda_n}) \xrightarrow{GH} (U_p, p_\infty , g_p) \subset (T_pX, p_\infty , g_p)
	\]
	after passing $\{\lambda_n\}$ to a suitable subsequence. Since $\overline{X}$ is complete, there exists a sequence of base points $\{q_n\} \subset \overline{X}$ where $q_n := f_n(p_n)$ such that $(\overline{X}, q_n, d_X) \xrightarrow{GH} (\overline{X}, q, d_X)$ for some $q \in \overline{X}$ after passing to a suitable subsequence again. Following the same idea of the proof of Theorem \ref{sqp} and applying Lemma \ref{upfc} on $\{f_n\}$ finish the proof.
\end{proof}

\begin{Remark}
	If we consider $PWT_p(X)$ instead of $WT_p(X)$, the above result can be even more strengthened. Assume that $X$ is homogeneously $\eta$-self-quasisymmetric at $p$, then $X$ is $\eta'$-quasisymmetrically embedded into every weak tangent in $PWT_p(X)$, where $\eta'(t) = 1/\eta^{-1}\!\!\left(\frac{1}{t}\right)$.
\end{Remark}

Self-quasisymmetricity is a rather restrictive property on metric spaces. The following definition is more practical and could be applied in many other fields. 

$X$ is a quasi-self-symmetric space if, roughly speaking, almost every small piece of $X$ can be uniformly quasisymmetrically mapped back into $X$ with a bounded size. Let $X$ be a metric space and $p$ be a point in $X$. $X$ is \emph{quasi-self-symmetric} at $p$ if there exist $r_p >0, L_p \geq 1$ such that for any $0 < r < r_p$, there exists an $\eta$-quasisymmetric map $f$ maps $B(p,r)$ into $X$ where
\begin{equation}\label{dist}
\frac{1}{L_p} \leq \diam(f(B(p,r))) \leq L_p.
\end{equation}
We call $r_p$ the \emph{quasi-self-symmetric scale} and $L_p$ the \emph{distortion bound} of $p$. Notice that they may not be unique.

We say a nonempty metric space $X$ is \emph{$\eta$-quasi-self-symmetric} if there exists a dense subset $X_{qss}$ of $X$ such that $X$ is $\eta$-quasi-self-symmetric at every point in $X_{qss}$ and the quasi-self-symmetric scales, the distortion bounds of these points are uniformly bounded. Thus we can define $r_0 = \sup_{p \in X_{qss}} r_p$ as the \emph{quasi-self-symmetric scale} of $X$ and $L_0 = \sup_{p \in X_{qss}} L_p$ as the \emph{distortion bound} of $X$.

We say a nonempty metric space $X$ is \emph{weakly $\eta$-quasi-self-symmetric} if there exists a dense subset $X_{qss}$ of $X$ such that $X$ is $\eta$-quasi-self-symmetric at every point in $X_{qss}$ and only the distortion bounds of these points are uniformly bounded.

In the following theorem, we show that quasi-self-symmetricity induces a converse implication, but not in the global sense.

\begin{Theorem}\label{wttob}
	Let $X$ be a compact, proper, doubling, $\eta$-quasi-self-symmetric space, $p$ be a point in $X$ and $T_pX$ be a weak tangent in $WT_p(X)$. Then there exists a ball $B(q, r)$ in $X$ and an $\eta'$-quasisymmetric embedding $f: B(q, r) \to T_pX$ such that $f(q) = p_\infty$ and $\eta'(t) = 1/\eta^{-1}\!\!\left(\frac{1}{t}\right)$. If $X$ is uniformly perfect, $r$ depends only on $X$ and $\eta$.
\end{Theorem}

\begin{proof}
	Let's assume that $(X, p_n, d_X/\lambda_n) \xrightarrow{GH} (T_pX, p_\infty , g_p)$. The proof of this theorem is analogous to the proof of Theorem \ref{QSweakT} and Theorem \ref{sqp}. We split it into two situations: either $X$ is uniformly perfect or not.
	
	Let $X$ be a uniformly perfect metric space. Without loss of generality, assume that $\lambda_n < r_0$ for all $n \in \mathbb{N}$. Here $r_0$ is the quasi-self-symmetric scale of $X$. Since $X$ is  uniformly perfect, there exists a point $x_n \in \overline{B}(p_n, \lambda_n)$ and a constant $C \leq 1$ depending only on $X$ such that 
	\begin{equation}\label{lespc}
	C \leq \frac{1}{\lambda_n}d_X(p_n, x_n) \leq 1
	\end{equation}
	for any $n \in \mathbb{N}$. Let $\{p'_n\}$ be a sequence of points at which $X$ is $\eta$-quasi-self-symmetric and $d_X(p'_n, p_n) < C/2 \cdot \lambda_n$. The reason we select $p'_n$ is that $X$ may not be quasi-self-symmetric at $p_n$ but $p_n$ can be approximated by such a point.
	
	\[
	\left(\overline{B}(p'_n, \lambda_n), p_n, \frac{d_X}{\lambda_n}\right) \xrightarrow{GH} (B_p, p_\infty , g_p)
	\]
	after passing $\{\lambda_n\}$ to a suitable subsequence, where $B_p$ is a subset of $T_pX$ due to  Lemma \ref{upfci}. Notice that $\diam(\overline{B}(p'_n, \lambda_n))/\lambda_n$ maybe less than $2$ because there may not exist any point near $\partial B(p'_n, \lambda_n)$. Since $d_X(p_n,x_n) \geq C \cdot \lambda_n$ and $d_X(p'_n, p_n) < C/2 \cdot \lambda_n$, we have that $C/2 \leq \diam(B(p'_n, \lambda_n))/\lambda_n \leq 2$ 
	
	Since $X$ is compact and quasisymmetries extend to the completions(Proposition $10.11$ in \cite{He01}), there exists a sequence of $\eta$-quasisymmetric maps $f_n: \left(\overline{B}(p'_n, \lambda_n), p_n, d_X/\lambda_n\right) \to (U_n, q_n, d_X)$ such that $1/L_0 \leq \diam(U_n) \leq L_0$, where $q_n := f_n(p_n)$, $U_n \subset X$ and $L_0$ is the distortion bound of $X$.
	
	Since $X$ and $\{U_n\}$ are compact, $U_n$ converges in the Hausdorff sense to a compact set $U_q \subset X$ after passing $\{U_n\}$ to a suitable subsequence. Readers may see Theorem $7.3.8$ in \cite{BBI01} for a reference. Furthermore, we may assume that $q_{n} \to q \in U_q$ after passing to a suitable subsequence again. Notice that (pointed) Hausdorff convergence implies (pointed) Gromov-Hausdorff convergence. It can be directly verified that $(U_{n}, q_{n}, d_X) \xrightarrow{GH} (U_q, q, d_X)$.
	
	Thus we have an $\eta$-quasisymmetry $f_n: \left(\overline{B}(p'_n, \lambda_n), p_n, \frac{d_X}{\lambda_n}\right) \to (U_n, q_n, d_X)$ for every $n \in \mathbb{N}$, where
	\[
	\left(\overline{B}(p'_n, \lambda_n), p_n, \frac{d_X}{\lambda_n}\right) \xrightarrow{GH} (B_p, p_\infty , g_p)
	\]
	and
	\[
	 (U_n, q_n, d_X) \xrightarrow{GH} (U_q, q, d_X).
	\]
	It remains to show that $f_n$ subconverges to an $\eta$-quasisymmetry.
	
	\[
	\xymatrix{
		\left(\overline{B}(p'_n, \lambda_n), p_n, \frac{d_X}{\lambda_n}\right) \ar[rr]^{f_n} \ar[dd]_{GH} &\ar@{=>}[dd]&  (U_n, q_n, d_X) \ar[dd]^{GH}
		\\
		\\
		(B_p, p_\infty , g_p) \ar[rr]_f  && (U_q, q, d_X).
	}
	\]
	
	Since
	\[
	C \leq \frac{1}{\lambda_n}d_X(p_n, x_n) \leq 1,
	\]
	it follows Proposition \ref{diam} that
	\[
	\frac{1}{2\eta\left( \frac{2}{C} \right)} \leq \frac{1}{2\eta\left( \frac{\diam(B(p'_n, \lambda_n))}{d_X(p_n,x_n)} \right)} \leq \frac{d_X(q_n, f_n(x_n))}{\diam(U_n)} \leq \eta\left( \frac{2d_X(p_n,x_n)}{\diam(B(p'_n, \lambda_n))} \right) \leq \eta\left( \frac{4}{C} \right).
	\]
	Thus
	\begin{equation}\label{lespcf}
	\frac{1}{2\eta\left( \frac{2}{C} \right) L_0} \leq d_X(q_n, f_n(x_n)) \leq \eta\left( \frac{4}{C} \right)L_0.
	\end{equation}
	
	Apply Lemma \ref{upfc} to $f_n$, we have that $\{f_n\}$ subconverges to an $\eta$-quasisymmetric map $f': (B_p, p_\infty) \to (U_q, q)$. Since $B_p \subset T_pX$, $U_q$ is $\eta'$-quasisymmetrically embedded into $T_pX$. Furthermore, if $r < 1/(2\eta\left( \frac{2}{C} \right) L_0)$, any $B(q,r)$ can be $\eta'$-quasisymmetrically embedded into $T_pX$, where this constant only depends on $X$ and $\eta$.
	
	If $X$ is not uniformly perfect, the proof is analogous to the second part of the proof of Theorem \ref{QSweakT}.
	
	Case $1$. If inequality \eqref{lespc} is true for some $C \leq 1$ and $\{x_n\}$, then we follows the above proof to get inequalities \eqref{lespcf}. Applying Lemma \ref{upfc} to $f_n$ finish the proof.
	
	Case $2$. If not, then $B_p$ is a singleton, which is trivial.
\end{proof}

Theorem \ref{wttoe} is a direct corollary of Theorem \ref{wttob}. We can also further strengthen Theorem \ref{wttob} by adding more conditions.

\begin{Remark}
	 One way to strengthen Theorem \ref{wttob} is to restrict quasi-self-symmetricity to local neighborhood. Assume that there exists an open neighborhood $U$ of $p$ such that $X$ is $\eta$-quasi-self-symmetric at a dense subset of $U$ where the quasi-self-symmetric scales and the distortion bounds of these points are uniformly bounded. Then for any $T_pX \in WT_p(X)$, there exists a ball $B(q, r)$ in $X$ and an $\eta'$-quasisymmetric embedding $f: B(q, r) \to T_pX$ such that $f(q) = p_\infty$ and $\eta'(t) = 1/\eta^{-1}\!\!\left(\frac{1}{t}\right)$. If $X$ is uniformly perfect, $r$ depends only on $X$ and $\eta$.
	
	Another way is to study the proper weak tangents. Assume that there exists a sequence of points $\{p_n\}$ such that $X$ is $\eta$-quasi-self-symmetric at $p_n$ for every $n$ and $p_n \to p$ in $X$. Furthermore, the quasi-self-symmetric scales and the distortion bounds of $\{p_n\}$ are uniformly bounded. Then for any $T_pX \in PWT_p(X)$, there exists a ball $B(q, r)$ in $X$ and an $\eta'$-quasisymmetric embedding $f: B(q, r) \to T_pX$ such that $f(q) = p_\infty$ and $\eta'(t) = 1/\eta^{-1}\!\!\left(\frac{1}{t}\right)$. If $X$ is uniformly perfect, $r$ depends only on $X$ and $\eta$.
\end{Remark}

The following theorem generalizes Theorem \ref{wttob} to weakly quasi-self-symmetric spaces. It shows that the restriction on the quasi-self-symmetric scale can be weakened.

\begin{Theorem}\label{wttobd}
	Let $X$ be a compact, proper, doubling, weakly $\eta$-quasi-self-symmetric space, $p$ be a point in $X$ and $T_pX$ be a weak tangent in $WT_p(X)$. Assume that there exists a neighborhood of $p$ and a constant $C$ depending on $p$ such that any point $p'$, at which $X$ is $\eta$-quasi-self-symmetric, in this neighborhood should  satisfy
	\[
	r_{p'} \geq C \cdot d_X(p,p')
	\]
	where $r_{p'}$ is the quasi-self-symmetric scale of $p'$. Then there exists a ball $B(q, r)$ in $X$ and an $\eta_1$-quasisymmetric embedding $f: B(q, r) \to T_pX$ such that $f(q) = p_\infty$ and $\eta_1$ depends only on $\eta$ and $X$. If $X$ is uniformly perfect, $r$ depends only on $X$ and $\eta$.
\end{Theorem}

\begin{proof}
	The proof is analogous to the proof of Theorem \ref{wttob}. The main difference here is that we need to select a suitable sequence of points at which $X$ is quasi-self-symmetric.
	
	We first assume that $X$ is uniformly perfect. Let $(X, p_n, \frac{d_X}{\lambda_n}) \xrightarrow{GH} (T_pX, p_\infty , g_p)$. We select $\{p'_n\}$ in the following way:
	\begin{enumerate}
		\item If $d_X(p_n,p) \geq (1+C)\lambda_n$, let $p'_n$ be a point at which $X$ is $\eta$-quasi-self-symmetric and $C_1 \lambda_n \leq d_X(p_n,p'_n) \leq C \lambda_n$ where $C_1$ is a constant comes from the uniform perfectness of $X$.\label{lpn}
		
		\item If $d_X(p_n,p) < (1+C)\lambda_n$, let $p'_n$ be a point at which $X$ is $\eta$-quasi-self-symmetric and $(2+C)\lambda_n \leq d_X(p_n ,p'_n) \leq C_2\lambda_n$, where $C_2$ is a constant comes from the uniform perfectness of $X$.\label{spn}
	\end{enumerate}
	Thus $r_{p'_n} \geq C \cdot d_X(p,p'_n) \geq C\lambda_n$ for any $n \in \mathbb{N}$.
	
	Let $B(p,r_0)$ be a ball inside the neighborhood stated in the theorem. Without loss of generality, assume that $C\lambda_n < r_0/2$, $C_1\lambda_n < r_0/2$ and $d_X(p_n,p) < r_0/2$ for all $n \in \mathbb{N}$. We denote by $B_n := \overline{B}(p'_n, C\lambda_n) \cup \{p_n\}$. Thus
	\[
	\left(B_n, p_n, \frac{d_X}{\lambda_n}\right) \xrightarrow{GH} (B_p, p_\infty , g_p)
	\]
	after passing $\{\lambda_n\}$ to a suitable subsequence and $(B_p, p_\infty , g_p) \subset T_pX$ by the isometric version of Lemma \ref{upfci}. Furthermore, $\overline{B}(p'_n, C\lambda_n)$ is $\eta$-quasisymmetric to $U'_n$, where $U'_n \subset X$, $1/L_0 \leq \diam(U'_n) \leq L_0$ and $L_0$ is the distortion bound of $X$. 
	
	We claim that for every $n \in \mathbb{N}$ there exists a point $q_n \in X$ and some constant $L$ depending only on $X$ such that $d_X(q_n, x') > L$ for every $x' \in U'_n$ . Such a $q_n$ exists if $L_0 < \diam(X)$. If not, we adjust $\overline{B}(p'_n, C\lambda_n)$ to a smaller ball, thus getting a smaller $U'_n$ under the same quasisymmetry. After adding $q_n$ as a base point of $U'_n$ if needed, we denote by
	\[
	U_n := \left\{
	\begin{array}{ll}
	U'_n & \textrm{if} \ p_n \in \overline{B}(p'_n, C\lambda_n);
	\\
	U'_n \cup q_n & \textrm{if} \ p_n \notin \overline{B}(p'_n, C\lambda_n).
	\end{array}
	\right.
	\]
	Let $f_n$ be the quasisymmetry that maps $\overline{B}(p'_n, C\lambda_n)$ to $U_n$. Then we extend $f_n$ to $B_n$ by defining $q_n:=f_n(p_n)$ when $p_n \notin \overline{B}(p'_n, C\lambda_n)$. Notice that there is no need to extend $f_n$ when $p_n \in \overline{B}(p'_n, C\lambda_n)$. We will prove that $f_n: B_n \to U_n$ is a quasisymmetry in next step.
	
	It is clear that if $p'_n$ is defined in case $\eqref{lpn}$, then $p_n \in \overline{B}(p'_n, C\lambda_n)$ and $B_n$ is $\eta$-quasisymmetric to $U_n$. Now let's consider when $p'_n$ is defined in case $\eqref{spn}$. In this case, $p_n \notin \overline{B}(p'_n, C\lambda_n)$ and $f_n$ is extended to $p_n$. Moreover, we have the following inequalities:
	\[
	\begin{aligned}
	&2\lambda_n \leq d_X(p_n, x) \leq (C+C_2) \lambda_n, \ \forall x \in \overline{B}(p'_n, C\lambda_n),
	\\
	&L \leq d_X(q_n,x') \leq \diam(X), \ \forall  x' \in U'_n,
	\\
	&2C_1 \lambda_n \leq \diam(\overline{B}(p'_n, C\lambda_n)) \leq 2C \lambda_n.
	\end{aligned}
	\]
	Let $x,y$ be any two points in $\overline{B}(p'_n, C\lambda_n)$, then
	\[
	\frac{1}{2\eta\left( \frac{\diam(\overline{B}(p'_n, C\lambda_n))}{d_X(x,y)} \right)} \leq \frac{d_X(f(x), f(y))}{\diam(U'_n)} \leq \eta\left( \frac{2d_X(x,y)}{\diam(\overline{B}(p'_n, C\lambda_n))} \right)
	\]
	by Proposition \ref{diam}. Thus
	\[
	\frac{1}{L_0}\frac{1}{2\eta\left( \frac{2C\lambda_n}{d_X(x,y)} \right)} \leq d_X(f(x), f(y)) \leq L_0 \eta\left( \frac{d_X(x,y)}{C_1\lambda_n} \right)
	\]
	and
	\[
	\frac{1}{L_0 \diam(X)}\frac{1}{2\eta\left(C \cdot \frac{d_X(x, p_n)}{d_X(x,y)} \right)} \leq \frac{d_X(f(x), f(y))}{d_X(f(x), q_n)} \leq \frac{L_0}{L} \eta\left( \frac{C+C_2}{C_1} \cdot \frac{d_X(x,y)}{d_X(x,p_n)} \right).
	\]
	It shows that $(B_n, p_n)$ is quasisymmetric to $(U_n, q_n)$ for every $n \in \mathbb{N}$. 
	
	Assume  that $U_n$ converges in the pointed Hausdorff sense to $U_q$ after passing to a suitable subsequence. Notice that from the definition of $\{p'_n\}$, we have that for every $n \in \mathbb{N}$,
	\[
	C_1 \leq \frac{1}{\lambda_n}d_X(p_n, p'_n) \leq C_2.
	\]
	It follows Proposition \ref{diam} that
	\[
	\frac{1}{2\eta\left( \frac{C+C_2}{C_1} \right)} \leq \frac{1}{2\eta\left( \frac{\diam(B_n)}{d_X(p_n,p'_n)} \right)} \leq \frac{d_X(q_n, f_n(p'_n))}{\diam(U_n)} \leq \eta\left( \frac{2d_X(p_n,p'_n)}{\diam(B_n)} \right) \leq \eta\left( \frac{C_2}{C_1} \right).
	\]
	Thus
	\[
	\frac{1}{2\eta\left( \frac{C+C_2}{C_1} \right) L_0} \leq d_X(q_n, f_n(p'_n)) \leq \eta\left( \frac{C_2}{C_1} \right)\diam(X).
	\]
	
	Applying Lemma \ref{upfc} to $f_n$, then we have that $B_p$ is $\eta_1$-quasisymmetric to $U_q$ where $\eta_1$ depends only on $\eta$ and $X$. Since $B_p \subset T_pX$, $U_q$ is $\eta_1$-quasisymmetrically embedded into $T_pX$. Furthermore, if $r < 1/(2\eta\left( \frac{C+C_2}{C_1} \right) L_0)$, any $B(q,r)$ can be $\eta_1$-quasisymmetrically embedded into $T_pX$, where this constant only depends on $X$ and $\eta$.
	
	If $X$ is not uniformly perfect, we follow the same idea in the proof of Theorem \ref{wttob} and finish the proof.
\end{proof}

\section{Gromov hyperbolic spaces and groups}\label{GHSG}

\subsection{Preliminaries}
A length space $(X, d)$ is called \emph{$\delta$-hyperbolic} (where $\delta \geq 0$) if for any triangle with geodesic sides in $X$, each side of the triangle is contained in the $\delta$-neighborhood of the union of two other sides. Recall that the \emph{Cayley graph} $\Gamma(G,S)$ of a finitely generated group $G$ with respect to a symmetric finite generating set $S$ is a graph whose vertex are elements of $G$ and $g_1,g_2 \in G$ is connected if and only if $g_1g_2^{-1} \in S$. A finitely generated group $G$ is called \emph{hyperbolic} if there exists a symmetric finite generating set $S$ for $G$ and a $\delta \geq 0$ such that the Cayley graph $\Gamma(G,S)$ of $G$ with respect to $S$ is $\delta$-hyperbolic. A group $G$ equipped with the \emph{word metric} with respect to $S$ is equivalent to its Cayley graph $\Gamma(G,S)$ equipped with the intrinsic metric(where each edge has length $1$). See \cite{BH99}, \cite{BBI01} and \cite{KB02} for a reference of Gromov hyperbolic spaces and groups.

Let $(X,d)$ be a metric space and $x,y,p \in X$. The \emph{Gromov product} of $x,y$ with respect to $p$ is defined by
\[
(x,y)_p = \frac{1}{2}\left( d(x,p)+d(y,p)-d(x,y) \right).
\]
In hyperbolic metric spaces, Gromov product measures how long two geodesics travel close together. More precisely, if $x, y, p$ are three distinct points in a $\delta$-hyperbolic space $X$, then the initial segments of length $(x, y)_p$ of any two geodesics connecting $x,p$ and $y,p$ are $2\delta$-close in Hausdorff distance.

Let $(X,d)$ be a proper $\delta$-hyperbolic metric space and $p$ be a chosen base point of $X$. We define the \emph{boundary at infinity} of $X$ by
\[
\partial_{\infty} X := \{[r]: r:[0,\infty) \to X \ \textrm{is a geodesic ray}, \ r(0) = p\},
\]
where two geodesic rays $r_1,r_2$ are equivalent if $d_H(r_1,r_2)<\infty$.

Equivalently, we also have
\[
\partial_{\infty} X := \{[\{x_n\}]: \{x_n\}_{n=1}^\infty \ \textrm{is a sequence converging to infinity in} \ X \},
\]
where $\{x_n\}$ converges to infinity if $\liminf\limits_{i,j \to \infty}(x_i, x_j)_p  = \infty$, and $\{x_n\}, \{y_n\}$ are equivalent if $\liminf\limits_{i,j \to \infty}(x_i, y_j)_p  = \infty$. It can be directly verified that the definitions do not depend on the base point.

We extend the Gromov product to $X \cup \partial_{\infty}X$ by
\[
(x ,y)_p := \sup \liminf_{i,j \to \infty} (x_i , y_j)_p
\]
where the supremum is taken over all sequences $\{x_i\}$ and $\{y_j\}$ in $X$ such that $x = [\{x_i\}]$ and $y = [\{y_i\}]$. When $x \in X$, $x = [\{x_i\}]$ means that $\lim_{i \to \infty} x_i = x$.

The following are some useful properties in studying the boundaries at infinity of hyperbolic spaces. See Remarks $3.17$ in Chapter III.H of \cite{BH99}.

\begin{Proposition}\label{gpb}
	Let $X$ be a $\delta$-hyperbolic space and $p$ be a chosen base point.
	\begin{enumerate}
		\item For any $x,y,z \in X \cup \partial_{\infty} X$, we have
		\begin{equation}
		(x,y)_p \geq \min\{  (x,z)_p, (y,z)_p \} -2 \delta.
		\end{equation}
		
		\item For any $x,y \in \partial_{\infty}X$ with $x = [\{x_i\}]$ and $y = [\{y_i\}]$, we have
		\begin{equation}
		(x, y)_p - 2\delta \leq \liminf_{i,j \to \infty} (x_i , y_j)_p \leq (x, y)_p.
		\end{equation}
	\end{enumerate}
\end{Proposition} 

There is a natural topology on $\partial_{\infty} X$, where the basis of this topology is the collection of
\begin{multline*}
V(x,\varepsilon) := \{y \in \partial_{\infty} X: \exists \ \textrm{gedeosic rays} \ r, r' \ \textrm{starting at} \ p
\\
\textrm{such that} \ [r]=x, [r']=y \ \textrm{and} \ \liminf_{t \to \infty}(r(t),r'(t))_p > \varepsilon\}
\end{multline*}
for every $x \in X$ and $\varepsilon>0$.

Equivalently, we also have
\begin{multline*}
V(x,\varepsilon) := \{y \in \partial_{\infty} X: \exists \ \textrm{sequences} \ \{x_n\}, \{y_n\}
\\
\textrm{such that} \ [\{x_n\}] = x, [\{y_n\}] = y \ \textrm{and} \ \liminf_{i,j \to \infty}(x_i,y_j)_p > \varepsilon\}
\end{multline*}
for every $x \in X$ and $\varepsilon >0$. Moreover, $\partial_{\infty} X$ is compact with this topology.

We say a metric $d_a$ on $\partial_{\infty} X$ is a \emph{visual metric} with respect to the base point $p$ and the visual parameter $a >1$ if there is $C_1, C_2 > 0$ such that the following holds:
\begin{enumerate}
	\item The metric $d_a$ induces the natural topology on $\partial_{\infty} X$;
	\item For any two distinct points $x, y \in \partial_{\infty} X$,
	\begin{equation}
	C_1a^{-(x,y)_p} \leq d_a(x,y) \leq C_2a^{-(x,y)_p}.
	\end{equation}
\end{enumerate}

We say two metric spaces $X$ and $Y$ are \emph{quasi-isometric} if there exists a map $f:X \to Y$ and $C_1 \geq 1, C_2 \geq 0$ such that
\begin{enumerate}
	\item For any two points $x_1,x_2 \in X$,
	\begin{equation}
		\frac{1}{C_1}d_X(x_1, x_2) - C_2 \leq d_Y(f(x_1), f(x_2)) \leq C_1 d_X(x_1, x_2) + C_2.
	\end{equation}
	
	\item For any $y \in Y$, there exists a $x \in X$ such that $d_Y(f(x), y) \leq C_2$.
\end{enumerate}

The following propositions illustrate several important properties of visual metrics.

\begin{Proposition}\label{hs}
	Let $X$ be a proper $\delta$-hyperbolic space, then
	\begin{enumerate}
		\item There exists an $a_0 > 1$ such that for any base point $p \in X$ and any $a \in (1, a_0)$, the boundary $\partial_{\infty} X$ admits a visual metric $d_a$ with respect to $p$ and $a$.
		
		\item Suppose $d'$ and $d''$ are visual metrics on $\partial_{\infty} X$ with respect to the same visual parameter $a$ and the base points $p'$ and $p''$, respectively. Then $d'$ and $d''$ are bi-Lipschitz equivalent.
		
		\item Suppose $d'$ and $d''$ are visual metrics on $\partial_{\infty} X$ with respect to the visual parameters $a'$ and $a''$ and the base points $p'$and $p''$, respectively. Then $d'$ and $d''$ are H\"older equivalent. 
	\end{enumerate}
\end{Proposition}

\begin{Proposition}
	Let $X$ and $Y$ be proper hyperbolic spaces and $f:X \to Y$ be a quasi-isometry, then $f$ induces a quasisymmetry $\hat{f}: \partial_{\infty} X  \to \partial_{\infty} Y$, where the boundaries at infinity are equipped with visual metrics.
\end{Proposition}
	
	We define the \emph{boundary at infinity of a hyperbolic group $G$} by $\partial_{\infty} G := \partial_{\infty} \Gamma(G,S)$. A group $G$ acts on a length space $X$ \emph{geometrically} if $G$ acts on $X$ as an isometry, i.e., every $g \in G$ acts on $X$ as an isometry, cocompactly, i.e.,  $X/G$ is compact, and properly discontinuously, i.e., for any compact $K \subset X$ the set $\{g \in G: gK \cap K = \emptyset\}$ is finite. For example, a finitely generated group $G$ acts on $\Gamma(G, S)$ geometrically for any symmetric finite generating set $S$. The following theorem studies geometric actions. Readers may refer to Proposition $8.19$ in Part I, Chapter $8$ of \cite{BH99} for a proof.
	
\begin{Theorem}[\v Svarc-Milnor]\label{be}
	Let $G$ be a group acting geometrically on a length space $X$. Then the group $G$ is finitely generated, the space $X$ is proper and for any symmetric finite generating set $S$ of G, there exists a quasi-isometry between $\Gamma(G,S)$ and $X$.
\end{Theorem}

Notice that every element $g$ of a hyperbolic group $G$ induces a quasisymmetry $\hat{g}: \partial_{\infty} G \to \partial_{\infty} G$.

\subsection{Metric structure of boundaries at infinity of Gromov hyperbolic groups}
Let $G$ be a Gromov hyperbolic group. The metric structure of $\partial_{\infty} G$ is an interesting object to study. For example, $\partial_{\infty} G$ is Ahlfors $Q$-regular for some $Q > 0$; thus $\partial_{\infty} G$ is doubling. See Section $15$ of \cite{KB02} for a reference. 

A metric space $X$ is called \emph{$H$-quasi-self-similar} if there exist $r_0 >0, H \geq 1$ such that given any ball $B$ with radius $r < r_0$, there exists a $H$-bi-Lipschitz map $f_B$ which maps $B$ into $X$ such that
\begin{equation}\label{dists}
\frac{1}{H} \frac{r_0}{r} d_X(x,y) \leq d_X(f_B(x), f_B(y)) \leq H \frac{r_0}{r} d_X(x,y)
\end{equation}
for all $x, y \in B$. In fact, it is a special case of quasi-self-symmetric.

The following theorem shows that $\partial_{\infty} G$ is quasi-self-similar.

\begin{Theorem}\label{hgbqss}
	Let $G$ be a hyperbolic group and $\partial_{\infty} G$ be its boundary at infinity. Then $(\partial_{\infty} G,d)$ is quasi-self-similar for any visual metric $d$ on $\partial_{\infty} G$.
\end{Theorem}

\begin{proof}
	Without loss of generality, we assume that in this proof the visual metric on $\partial_{\infty} G$ is with respect to $1 \in G$; otherwise, applying part $(2)$ of Proposition \ref{hs} induces the desired result.
	
	We first prove the theorem when $G$ is a free group $G : =F(a_1,\ldots,a_r)$.
	
	Let $x,y$ be reduced words(i.e., the simplest representation in the generators) on $G\cup \partial_{\infty} G$, and let the Gromov product of them with respect to the identity element $1$ be $(x,y)_1 := t$. If given the representations of $x,y$, the maximal common initial segment of $x,y$ is the maximal common part, starting from initial, of their representations. Thus $t$ is the number of the maximal common initial segment of $x,y$.
	
	For any visual metric $d$ on $\partial_{\infty} G$, there exist $C_1, C_2 >0, a>1$ such that
	\[
	C_1 \cdot a^{-(x,y)_1} \leq d(x,y) \leq C_2 \cdot a^{-(x,y)_1}
	\]
	for every $x,y\in \partial_{\infty} G$. Notice that $a \in (1, \infty)$ since the Cayley graph of $G$ is a tree.
	
	Let $p$ be a point in $\partial_{\infty} G$. We think of $p$ as a semi-infinite(bounded in one direction) reduced word in $F(a_1,\ldots,a_r)$. For any integer $m\geq 0$, we denote by 
	\[
	U(p,m) := \{q \in \partial_{\infty} G: (p,q)_1 > m\}.
	\]
	This is well defined since every element in $\partial_{\infty} G$ has a unique semi-infinite reduced word.
	
	Take any $p\in \partial_{\infty} G$ and any integer $m\geq 1$. Let $w \in F(a_1,\ldots,a_r)$ be the element given by the initial segment of $p$ of length $m$ and let $g=w^{-1}$.
	
	Let $x,y\in U(p,m)$. Then, viewed as semi-infinite reduced word in $G$, both $x$ and $y$ have $w$ as their initial segment of length $m$, so $x=wx'$ and $y=wy'$ where $x',y'$ are semi-infinite reduced words and thus elements of $\partial_{\infty} G$.
	
	We have $gx=w^{-1}wx'=x'$ and $gy=w^{-1}wy'=y'$. Let $v$ be the maximal common initial segment of $x',y'$, then $wv$ is the maximal common initial segment of $x,y$. Let $|v|$ denotes the length of $v$ under the word metric. Since $G$ is a free group, the Cayley graph of $G$ is a tree and we have $(x', y')_1=|v|$ and $(x, y)_1=|v|+m$.
	
	Hence 
	\[
	C_1 \cdot \frac{1}{a^{|v|}} \leq d(gx,gy)=d(x',y') \leq C_2 \cdot \frac{1}{a^{|v|}}
	\]
	and
	\[
	C_1 \cdot \frac{1}{a^{|v|+m}} \leq d(x,y) \leq C_2 \cdot \frac{1}{a^{|v|+m}}.
	\]
	Thus
	\[
	\frac{C_1}{C_2}a^m d(x,y) \leq d(gx,gy) \leq \frac{C_2}{C_1}a^m d(x,y).
	\]
	Notice that
	\[
	B(p, C_1 \cdot a^{-m}) \subset \{q \in \partial_{\infty} G: (p,q)_1 > m\} = U(p,m).
	\]
	Let $B(p,r)$ be a ball and $C_1 \cdot a^{-(l+1)} \leq r \leq C_1 \cdot a^{-l}$, then
	\[
	\frac{C_1^2}{a C_2}\frac{1}{r} d(x,y) \leq \frac{C_1}{C_2}a^l d(x,y) \leq d(gx,gy) \leq \frac{C_2}{C_1}a^l d(x,y) \leq C_2\frac{1}{r} d(x,y)
	\]
	for any $x,y \in B(p,r)$. So $\partial_{\infty} G$ is quasi-self-similar.
	
	If $G$ is an arbitrary hyperbolic group, the argument is similar. Let $\Gamma(G,S)$ be a Cayley graph of $G$ and $d$ be any visual metric on $\partial_{\infty} G$ with respect to the base point $1$ and the visual parameter $a$. We assume that $\Gamma(G,S)$ is $\delta$-hyperbolic. There exist $C_1, C_2 > 0$ such that for any $x,y\in \partial_{\infty} G$, 
	\[
	C_1 \cdot a^{-(x,y)_1}\le d(x,y) \le C_2 \cdot a^{-(x,y)_1}.
	\]
	
	For any $m\geq 0$ and $p \in \partial_{\infty} G$, we define
	\begin{multline*}
	U(p,m) := \{q \in \partial_{\infty} G : \exists \ \textrm{geodesic rays} \ r, r' \ \textrm{starting at} \ 1  
	\\
	\textrm{such that} \ [r]=p, [r']=q \ \textrm{and} \ \liminf_{t \to \infty}(r(t),r'(t))_1 > m + 2\delta\}.
	\end{multline*}
	
	Notice that $U(p,m)$ is a neighborhood of $p$ in $\partial_{\infty} G$ (in fact, it is a basic neighborhood in the original definition of the topology on $\partial_{\infty} G$).
	
	Take any $p \in \partial_{\infty} G$ and any integer $m\geq 1$. Let $r_p$ be a geodesic ray representing $p$ and $w \in G$ be the element given by the initial segment of $r_p$ of length $m$. Let $g := w^{-1}$.
	
	Let $x,y \in U(p,m)$. We may assume that, by Proposition \ref{gpb}, there exist two geodesic rays $r_x$ and $r_y$ representing $x$ and $y$ in $U(p,m)$, respectively, such that
	\[
	\liminf_{t \to \infty}(r_x(t),r_p(t))_1 \geq m \qquad \textrm{and} \qquad \liminf_{t \to \infty}(r_y(t),r_p(t))_1 \geq m.
	\]
	Thus $(x,p)_1 \geq m$ and $(y,p)_1 \geq m$.
	
	Let $d_G$ denotes the word metric on $G$ with respect to $S$. Then 
	\begin{eqnarray*}
		&&(r_x(t), r_y(t))_1 - (g r_x(t), g r_y(t))_1
		\\
		& = & \frac{1}{2} \left( d_G(r_x(t), 1) - d_G(g r_x(t), 1) + d_G(gr_y(t), 1) - d_G(r_y(t), 1) \right).
	\end{eqnarray*}
	for every $t$.
	
	Let $t$ be sufficiently large, i.e., $d_G(r_x(t), 1) \geq m$ and $d_G(r_y(t),1) \geq m$. Notice that, by the definition of Gromov product, the initial segments of length $m$ of $r_x$ and $r_p$ are $2\delta$-close in Hausdorff distance. Then there exists a number $t'$ such that $d_G(r_x(t'),1) \leq m$ and $d_G(r_x(t'), w) \leq 2 \delta$.
	Thus, by analyzing the geodesic triangles, for sufficiently large $t$ we have
	\begin{equation*}
	d_G(r_x(t), 1) - d_G(g r_x(t), 1) = d_G(r_x(t), 1) - d_G(r_x(t), w) \leq m
	\end{equation*}
	and
	\begin{eqnarray*}
	d_G(r_x(t), 1) - d_G(g r_x(t), 1) &=& d_G(r_x(t), 1) - d_G(r_x(t), w) 
	\\
	&=& d_G(r_x(t), r_x(t')) - d_G(r_x(t), w) + d_G(r_x(t'), 1)
	\\
	& \geq& m - 4\delta.
	\end{eqnarray*}
	
	\begin{figure}[htbp]
		\centering
		\includegraphics[width=2in]{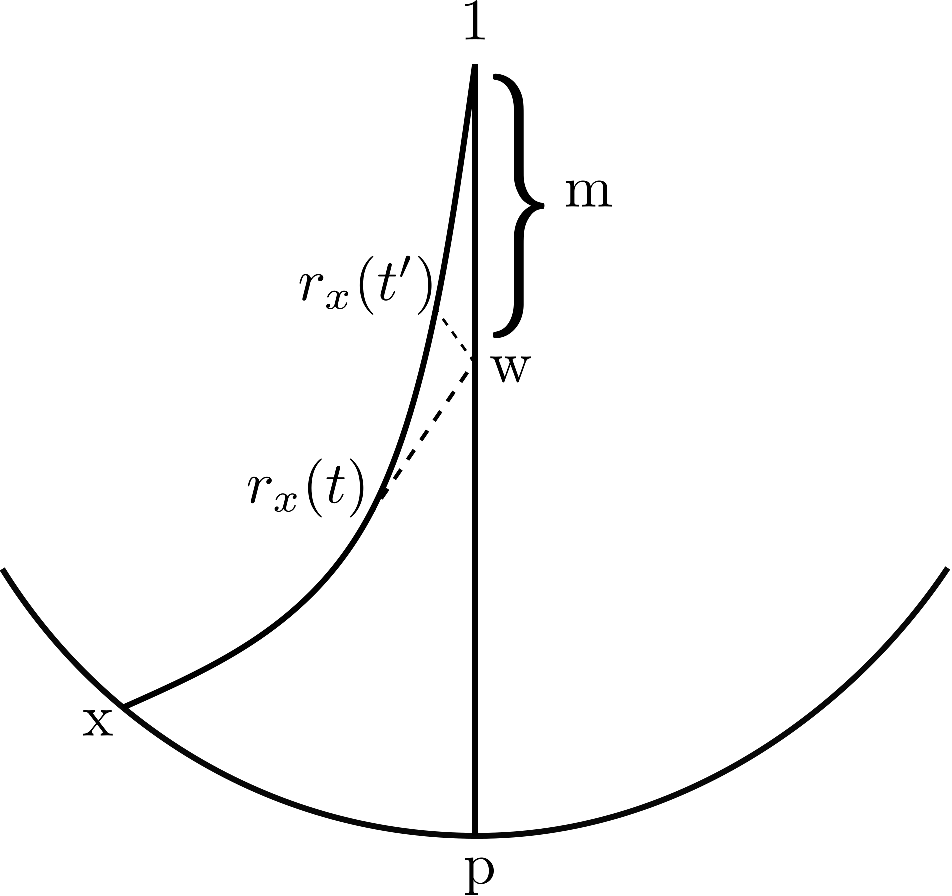}
		\caption{Geodesic triangles that involve $1$, $w$, $r'(t)$ and $r(t)$.}
	\end{figure}
	
	Similarly, for sufficiently large $t$ we also have
	\begin{equation}\label{gty}
	m - 4\delta \leq d_G(r_y(t), 1) - d_G(g r_y(t), 1) = d_G(r_y(t), 1) - d_G(r_y(t), w) \leq m.
	\end{equation}
	This finally shows that
	\begin{equation}\label{egr}
	m - 4\delta \leq (r_x(t), r_y(t))_1 - (g r_x(t), g r_y(t))_1 \leq m
	\end{equation}
	for sufficiently large $t$.
	
	By Proposition \ref{gpb}, we have that
	\begin{equation}\label{gtb}
	|(x,y)_1 - \liminf_{t \to \infty}(r_x(t),r_y(t))_1| \leq 2\delta.
	\end{equation}
	
	Similarly, $gr_x, gr_y$ are two geodesic rays representing $gx$ and $gy$ on $\partial_{\infty} G$, respectively. Thus
	\begin{equation}\label{agtb}
	|(gx,gy)_1 - \liminf_{t \to \infty}(gr_x(t),gr_y(t))_1| \leq 2\delta.
	\end{equation}
	
	Combining inequalities \eqref{egr} with \eqref{gtb} and \eqref{agtb}, we have
	\[
	m - 8\delta \leq (x,y)_1-(gx,gy)_1 \leq m + 4\delta.
	\]
	Namely, $(x,y)_1$ equals $m+(gx,gy)_1$ up to an additive error of $8 \delta$.
	
	Hence
	\[
	d(gx,gy) \leq C_2 \cdot a^{-(gx,gy)_1} \leq C_2 \cdot a^{m+4\delta} a^{-(x,y)_1} \leq \frac{C_2}{C_1} \cdot a^{m+4\delta}d(x,y),
	\]
	and
	\[
	d(gx,gy) \geq C_1 \cdot a^{-(gx,gy)_1} \geq C_1 \cdot a^{m-8\delta} a^{-(x,y)_1} \geq \frac{C_1}{C_2} \cdot a^{m-8\delta}d(x,y).
	\]
	Thus
	\[
	\frac{C_1}{C_2} \cdot a^{m-8\delta}d(x,y) \leq d(gx,gy) \leq \frac{C_2}{C_1} \cdot a^{m+4\delta}d(x,y).
	\]
	Notice that
	\[
	B(p, C_1 \cdot a^{-(m+4\delta)}) \subset \{q \in \partial_{\infty} G: (p,q)_1\ge m + 4\delta\} \subset U(p,m).
	\]
	Let $B(p,r)$ be a ball and $C_1 \cdot a^{-(l+4\delta+1)} \leq r \leq C_1 \cdot a^{-(l+4\delta)}$, then there exists a $f$ such that
	\[
	\frac{C_1^2}{a^{1+12\delta} C_2}\frac{1}{r} d(x,y) \leq \frac{C_1}{C_2} \cdot a^{l-8\delta}d(x,y) \leq d(fx,fy) \leq \frac{C_2}{C_1} \cdot a^{l+4\delta}d(x,y) \leq C_2\frac{1}{r} d(x,y)
	\]
	for any $x,y \in B(p,r)$. So $\partial_{\infty} G$ is quasi-self-similar. Thus finishes the proof.
\end{proof}

\begin{Remark}	
	
	Let $H$ be a finite index subgroup of $G$. Then $H$ is hyperbolic if and only if $G$ is hyperbolic and in this case $\partial_{\infty} H = \partial_{\infty} G$. A group $G$ is \emph{virtually free} if there exists a finite index subgroup of $G$ which is free. The boundary at infinity of a virtually free group is not only quasi-self-similar, but also admit a stronger structure. Let $H$ be the free subgroup of $G$ and we denote by $d(x,y) = 2^{-(x,y)_1}$ for any $x,y  \in \partial_{\infty} H$. It can be directly verified that $d$ is a visual metric since the Cayley graph of $H$ is a tree. Given a point $p \in \partial_{\infty} H$, we denote by $w$ the element given by the initial segment of $p$ of length $1$ and let $g:=w^{-1}$.  It follows the proof of Theorem \ref{hgbqss} that $d(gx,gy) = 2d(x,y)$ for any $x,y \in B(p, 1/2)$.
\end{Remark}

Let $X$ be a proper hyperbolic metric space and $p$ be a chosen base point. Let $G$ be a group acting on $X$ by isometries. We say $\Lambda(G)$ is the \emph{limit set} of $G$ on $X$ if
\begin{multline*}
\Lambda(G) := \{[\{g_n p\}] \in \partial_{\infty} X: \exists \ g_n \in G \ \textrm{such that} \ 
\\
\{g_n p\}_{n=1}^\infty \ \textrm{is a sequence converging to infinity in} \ X \}.
\end{multline*}
This definition is still independent of the base point. One can also define the \emph{conical limit set} $\Lambda_c(G)$ as the collection of points $q \in \Lambda(G)$ which are approximated by a sequence from the orbit $G p$ such that this sequence is contained in a bounded neighborhood of some geodesic ray $r$ with $[r] = q$. In this situation, we also get a corresponding translation action of $G$ on $\Lambda_c(G)$.

In next step, we are trying to generalize Theorem \ref{hgbqss} to the conical limit set. However, the idea in Theorem \ref{hgbqss} meets obstacles when dealing with the conical limit set. We transfer to a new method: Expanding cover.

A sequence of maps $f_i: U_i \to Y$ is called an \emph{expanding cover} of $X$ if there exists a constant $L_i > 1$ such that
\[
d_X(f_i(x_1),f_i(x_2)) \geq L_i \cdot d_X(x_1, x_2)
\]
for any $x_1, x_2 \in U_i$ and $\bigcup_i U_i =X$.
If there exists a ``controlled" expanding cover on a metric spaces, then it is quasi-self-similarity. The following theorem follows the same idea of the distortion lemma on p.$42$ of \cite{Su82}. For the sake of completeness, we outline a proof of it.

\begin{Theorem}\label{epc}
	Let $X$ be a compact metric space and $\{U_i\}_{i=1}^N$ be a finite open cover of $X$. If for each $i$, there exists $f_i: U_i \to X$, $L_i > 1$, $\alpha_i \geq 1$, and $C_i > 0$ such that
	\begin{enumerate}
		\item $d_X(f_i(x),f_i(y)) \geq L_i \cdot d_X(x,y)$\label{expand};
		\vspace{0.1in}
		\item $\frac{d_X(f_i(x),f_i(y))}{d_X(x,y)}-\frac{d_X(f_i(z),f_i(w))}{d_X(z,w)} \leq C_i \cdot \diam(\{x,y,z,w\})^{\alpha_i}$\label{control};
	\end{enumerate}
	for every $x,y,z,w \in U_i$, then $X$ is quasi-self-similar.
\end{Theorem}

\begin{proof}
	Let $r_0$ be the Lebesgue number of $\{U_i\}_{i=1}^N$, i.e., a number that any subset of $X$ with diameter less than it is contained in one of $U_i$. For any $B(x, r)$ with $r<r_0$, $B(x,r)$ should be a subset of one of $\{U_i\}_{i=1}^N$. Applying the corresponding $f_i$ to $B(x,r)$ maps  it to a larger image. Iterating this action until we get an image whose diameter is no smaller than $r_0$.
	
	Suppose that $B(x,r)=B_{i_1} \subset U_{i_1}$, and after iterating $j$ times we expand $B_{i_1}$ to $B_{i_j} \subset U_{i_j}$. We assume that $B_{i_n}$ is the first one whose diameter is larger than $r_0$. It is sufficient to prove that there exists a constant $H > 1$ independent of $n$ such that
	\[
	 \frac{d_X(x_n,y_n)}{d_X(x_1,y_1)} \leq H \cdot \frac{d_X(x'_n,y'_n)}{d_X(x'_1,y'_1)}
	\]
	where $x_1,y_1,x'_1,y'_1 \in B_{i_1}$ and $x_j, y_j, x'_j,y'_j$ are the corresponding images of them in $B_{i_j}, j = 1, \ldots, n$. Since if we let $d_X(x_1',y_1') \geq r/2$, then
	\[
	\frac{d_X(x_n,y_n)}{d_X(x_1,y_1)} \leq H \cdot \frac{d_X(x'_n,y'_n)}{d_X(x'_1,y'_1)} \leq 2H \frac{\diam(X)}{r_0} \frac{r_0}{r}.
	\]
	Moreover, if we let $d_X(x_n,y_n) \geq r_0/2$, then
	\[
	 \frac{d_X(x'_n,y'_n)}{d_X(x'_1,y'_1)} \geq \frac{1}{H}\frac{d_X(x_n,y_n)}{d_X(x_1,y_1)} \geq \frac{1}{2H} \frac{r_0}{r}.
	\]
	We denote by $D_j := \diam(\{x_j, y_j, x'_j,y'_j\})$.
	
	By conditions \eqref{expand} and \eqref{control}, we have
	\[
	\left(\frac{d_X(f_i(x),f_i(y))}{d_X(x,y)}-\frac{d_X(f_i(z),f_i(w))}{d_X(z,w)}\right)/\left( \frac{d_X(f_i(z),f_i(w))}{d_X(z,w)} \right) \leq \frac{C_i}{L_i} \cdot \diam(\{x,y,z,w\})^{\alpha_i},
	\]
	which implies
	\begin{eqnarray*}
		\frac{d_X(f_i(x),f_i(y))}{d_X(x,y)}/\frac{d_X(f_i(z),f_i(w))}{d_X(z,w)} & \leq & 1+\frac{C_i}{L_i} \cdot \diam(\{x,y,z,w\})^{\alpha_i}
		\\
		& \leq & e^{\lambda \cdot \diam(\{x,y,z,w\})}
	\end{eqnarray*}
	for any $x,y,z,w \in U_i$ and for any $i$. Here $\lambda$ is a constant which depends on $\{L_i\}, \{\alpha_i\}$ and $\{C_i\}$.
	
	\begin{eqnarray*}
		&&\log\left( \frac{d_X(x_n,y_n)}{d_X(x_1,y_1)}/\frac{d_X(x'_n,y'_n)}{d_X(x'_1,y'_1)} \right)  
		\\
		& =  & \ \sum_{j=1}^{n-1}\log\left( \left( \frac{d_X(x_{j+1},y_{j+1})}{d_X(x_{j},y_{j})} \right)/\left( \frac{d_X(x'_{j+1},y'_{j+1})}{d_X(x'_{j},y'_{j})} \right) \right)
		\\
		& \leq & \sum_{j=1}^{n-1} \log \left( e^{\lambda D_{j}} \right)
		\\
		& \leq & \lambda \sum_{j=1}^{n-1} D_{j}.
	\end{eqnarray*}
	
	Let $L = \min_i\{L_i\}$, then $D_j \leq \diam(B_{i_j}) \leq \diam(B_{i_{j+1}})/L$ and $\diam(B_{i_n}) \leq \diam(X)$. Thus $\{D_j\}$ is bounded by a geometric series and we finish the proof.
\end{proof}

We claim that there exists an expanding cover on the conical limit set.

\begin{Corollary}\label{ls}
	Let $(X, d_X)$ be a proper hyperbolic metric space, $G$ be a finitely generated group acting on $X$ by isometries and $\Lambda_c(G)$ be the conical limit set of $G$ on $X$. Then there exists an expanding cover of $\Lambda_c(G)$ for any visual metric on $\partial_{\infty} X$.
\end{Corollary}

\begin{proof}
The proof of Corollary \ref{ls} is in the same idea as the proof of Theorem \ref{hgbqss}.

We fix a base point $p$. Taking any $q \in \Lambda_c(G)$ and denote by $[q_n]:=q$ where $\{q_n\} \in Gp$. Since $q \in \Lambda_c(G)$ , we may assume that there exists a geodesic ray $\gamma_q$ such that $\gamma_q(0) =p$, $[\gamma_q] =q$ and the distance between $q_n$ and $\gamma_q$ is bounded by $H$. Let $q_i$ be an element in $\{q_n\}$ such that $m \geq d_X(q_{i}, p) \geq m-1$ for some integer $m$ where $m > H + 8\delta$. We denote by
\begin{multline*}
U(q,m) := \{x \in \Lambda_c(G): \exists \ \textrm{some sequence} \ \{q_n\}, \{x_n\}
\\
\textrm{such that} \ [\{q_n\}] = q, [\{x_n\}] = x \ \textrm{and} \ \liminf_{i,j \to \infty}(q_i,x_j)_p \geq m+2\delta\}.
\end{multline*}

Let $w \in G$ be the element given by $q_{i} := w p$ and we denote by $g := w^{-1}$.

Let $x,y \in U(q,m)$. We may assume that there exist two geodesic rays $r_x$ and $r_y$ representing $x$ and $y$ in $U(p,m)$, respectively, such that
\[
\liminf_{t \to \infty}(r_x(t),r_q(t))_p \geq m \qquad \textrm{and} \qquad \liminf_{t \to \infty}(r_y(t),r_q(t))_p \geq m.
\]
Thus $(x,p)_p \geq m$ and $(y,p)_p \geq m$.

Then 
\begin{eqnarray*}
	&&(r_x(t), r_y(t))_p - (g r_x(t), g r_y(t))_p
	\\
	& = & \frac{1}{2} \left( d_X(r_x(t), p) - d_X(g r_x(t), p) + d_X(gr_y(t), p) - d_X(r_y(t), p) \right).
\end{eqnarray*}
for every $t$.

Let $t$ be sufficiently large, i.e., $d_X(r_x(t), p) \geq m+1$ and $d_X(r_y(t),p) \geq m+1$. Notice that, by the definition of Gromov product, the initial segments of length $m$ of $r_x$ and $r_p$ are $2\delta$-close in Hausdorff distance. Then there exists a number $t'$ such that $d_G(r_x(t'),p) \leq m$ and the distance between $r_x(t')$ and $r_q$ is less or equal than $2\delta$. Thus, by analyzing the geodesic triangles, for sufficiently large $t$ we have
\begin{eqnarray*}
	d_X(r_x(t), p) - d_X(g r_x(t), p) &=& d_X(r_x(t), p) - d_X(r_x(t), q_i) 
	\\
	&=& d_X(r_x(t), r_x(t')) - d_X(r_x(t), q_i) + d_X(r_x(t'), p)
	\\
	& \geq& m - H - 4\delta.
\end{eqnarray*}

\begin{figure}[htbp]
	\centering
	\includegraphics[width=2in]{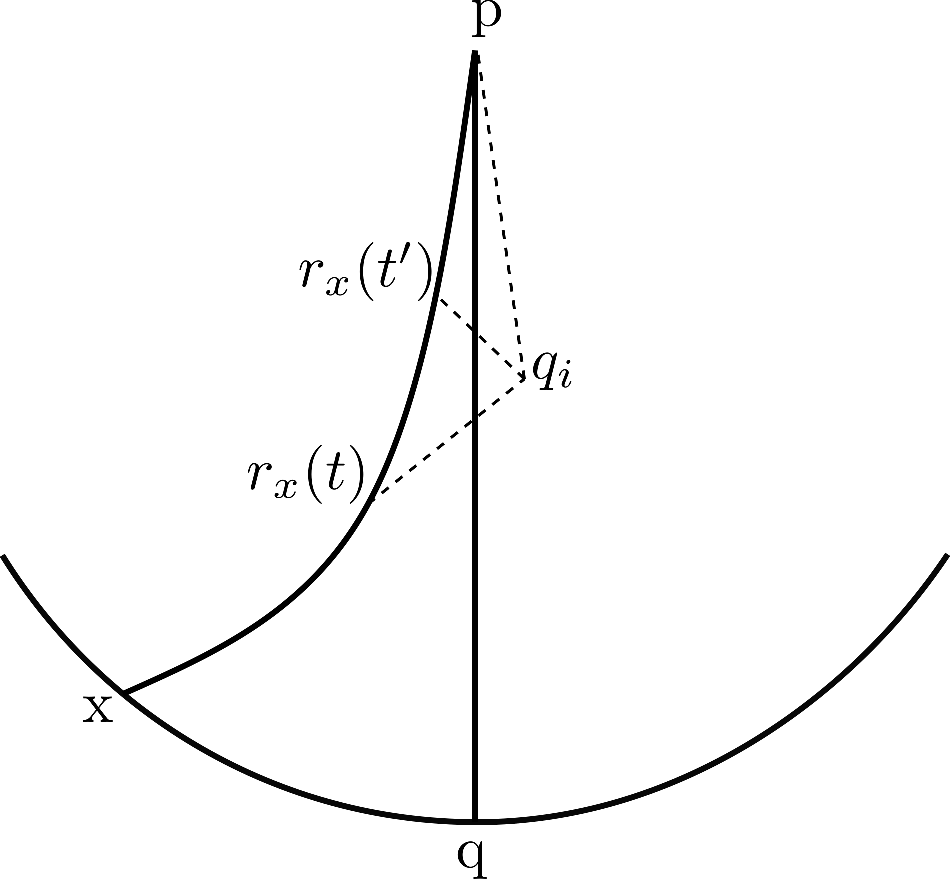}
	\caption{Geodesic triangles that involve $p$, $q_i$, $r'(t)$ and $r(t)$.}
\end{figure}

Similarly, for sufficiently large $t$ we also have
\begin{equation*}
d_X(r_y(t), p) - d_X(g r_y(t), p) \geq m -H - 4\delta.
\end{equation*}
This finally shows that
\begin{equation*}
(r_x(t), r_y(t))_p - (g r_x(t), g r_y(t))_p \geq m - H - 4\delta
\end{equation*}
for sufficiently large $t$.

Thus
\[
(x,y)_p-(gx,gy)_p \geq m - H - 8\delta.
\]
and there exists a constant $C > 1$ depends only on $\Lambda_c(X)$ such that
\[
d(gx,gy) \geq C \cdot a^{m-H - 8\delta}d(x,y).
\]

Then we finish the proof.
\end{proof}

To ``control" the above expanding cover, it requires more structure than just hyperbolicity. We introduce some new notions now.

We denote by $M_k^2$ the $2$-dimensional complete, simply connected, Riemannian manifold of constant sectional curvature $K \in \mathbb{R}$. It is unique up to isometry. We denote by $D_k$ the diameter of $M_k^2$. Readers may refer to Part I, Chapter $2$ of \cite{BH99} for more details.

Let $\Delta$ be a geodesic triangle in a length space $X$, i.e., a triangle with geodesic segments as its sides. $\Delta$ is called satisfying \emph{CAT$(k)$ inequality} if the following holds: Suppose $x, y, z \in X$ is the vertice of $\Delta$ and let $p,q$ be points on the sides connecting $x,y$ and $y,z$, respectively. Let $x',y',z' \in M_k^2$ such that $d(x,y) = d(x',y'), d(y,z) = d(y',z')$ and $d(x,z) = d(x',z')$. If $p',q'$ be any points on the geodesic segments connecting $x',y'$ and $y',z'$, respectively, such that $d(x, p) = d(x', p')$ and $d(x, q)  =d(x', q')$, then $d(p, q) \leq  d(p', q')$.

A length space $X$ is called \emph{CAT$(k)$} if every geodesic triangle with diameter less than $2D_k$ satisfies the CAT$(k)$ inequality when $k >0$ and every geodesic triangle satisfies the CAT$(k)$ inequality when $k \leq 0$.

Notice that any CAT$(k)$ space with $k < 0$ is also hyperbolic.

In \cite{Bo96}, Bourdon shows that if X is a proper CAT($-b^2$) space, then for each number $a \in (1, e^b]$ and each $p \in X$, the formula $d_p(\xi,\xi') := a^{-(\xi,\xi')_p}$ defines a metric on $\partial_\infty X$.

Let $x,y \in X$, $\xi \in \partial_{\infty} X$ and $\gamma$ be a geodesic ray representing $\xi$. We denote by the \emph{Busemann function}
\[
B_\xi(x) := \lim_{t \to \infty} \left(d_X(x, \gamma(t)) - t \right)
\]
and
\[
B_\xi(x,y) : = B_\xi(x) - B_\xi(y) = \lim_{t \to \infty}\left(d_X(x, \gamma(t))-d_X(y, \gamma(t))\right).
\]
Notice that these definitions are independent of the choice of $\gamma$.

The following two results are selected from \cite{Bo93} , which give a ``conformal" structure on the limit set. Readers may refer to Proposition $2.6.1$ and Corollary $2.6.3$ in \cite{Bo93} for more details.

\begin{Lemma}[Bourdon]\label{cscs}
	Let $X$ be a proper CAT($k$) space with $k < 0$, $p$ be the chosen base point, $g \in \mathrm{Isom}(X)$ and $\xi,\xi'$ be any two points on $\partial_\infty X$. There exists a visual metric $d$, where $a$ is the visual parameter of $d$, on $\partial_\infty X$, such that
	\[
	\frac{d(g\xi,g\xi')}{d(\xi,\xi')} =   a^{\frac{1}{2}\left(B_\xi(p,g^{-1}p)+B_{\xi'}(p,g^{-1}p)\right)}.
	\]
\end{Lemma}

\begin{Lemma}[Bourdon]\label{csmls}
	Let $X$ be a proper hyperbolic space and $p$ be the chosen base point. Let $d$ be any visual metric defined on $\partial_\infty X$, then for any $x, y$ in $X$ the function on $(\partial_{\infty} X,d)$, defined by:
	\[
	\xi \longrightarrow B_{\xi}(x,y)
	\]
	is Lipschitz.
\end{Lemma}
\begin{Remark}
	The original results in \cite{Bo93} are proved for a specifically chosen visual metric on the boundary at infinity of CAT($-1$) spaces; however, we can generalize it to CAT($k$) spaces with $k <0$ but still follow the same proof. The main reason is that for each number $a \in (1, e^b]$ and each $p \in X$, $d_p(\xi,\xi') := a^{-(\xi,\xi')_p}$ defines a metric on $\partial_\infty X$.
\end{Remark}

\begin{Theorem}\label{cls}
	Let $(X, d_X)$ be a proper CAT($k$) space with $k < 0$, $p$ be a chosen base point, $G$ be a finitely generated group acting on $X$ by isometries and $\Lambda_c(G)$ be the conical limit set of $G$ on $X$. If $\Lambda_c(G)$ is compact, then $\Lambda_c(G)$ is quasi-self-symmetric for any visual metric on $\partial_{\infty} X$.
\end{Theorem}

\begin{proof}	
	Let $d$ be the visual metric on $\partial_{\infty} X$ defined in Lemma \ref{cscs} and $a$ is the visual parameter. Since any two visual metrics are quasisymmetric by Proposition \ref{hs},  it is sufficient to prove that $(\Lambda_c(G),d)$ is quasi-self-similar. Let $g$ and $U(q,m)$ be the notation defined in Corollary \ref{ls}. It comes from Theorem \ref{epc} and Corollary \ref{ls} that we only need to prove the following result:  There exists a constant $C>0$ depends only on $g$ and $U(q,m)$ such that
	\[
	\frac{d(gx,gx')}{d(x,x')} - \frac{d(gy,gy')}{d(y,y')} \leq  C \cdot \diam(x,x',y,y')
	\]
	for any $x,x',y,y' \in U(q,m)$.
	
	\[
	\begin{aligned}
	\frac{d(gx,gx')}{d(x,x')} - \frac{d(gy,gy')}{d(y,y')} & = a^{\frac{1}{2}\left(B_x(p,g^{-1}p)+B_{x'}(p,g^{-1}p)\right)} - a^{\frac{1}{2}\left(B_y(p,g^{-1}p)+B_{y'}(p,g^{-1}p)\right)}
	\\
	& \leq C_1\left|\left(B_x(p,g^{-1}p), B_{x'}(p,g^{-1}p)\right)-\left(B_y(p,g^{-1}p), B_{y'}(p,g^{-1}p)\right)\right|
	\\
	& \leq C_1 C_2 \sqrt{d(x,y)^2 + d(x',y')^2}
	\\
	& \leq C \cdot \diam(x,x',y,y')
	\end{aligned}
	\]
	The first inequality comes from the mean value theorem for function $f(x,y) = a^{\frac{1}{2}(x+y)}$ and $a^{\frac{1}{2}\left(B_x(p,g^{-1}p)+B_{x'}(p,g^{-1}p)\right)}$ is uniformly bounded and the second inequality comes from Lemma \ref{csmls}.
\end{proof}

In the rest of this section, we always assume that the boundary at infinity of a hyperbolic space contains more than two points(i.e., non-elementary).

If $X$ is any compact metrizable space that has at least three points, we denote the space of distinct triples of $X$ by $\Sigma_3(X)$. Namely,
\[
\Sigma_3(X) = \{(o,p,q) \in X^3: o \neq p, o \neq q, p \neq q \}.
\]
Assume that a group $G$ acts on $X$ by homeomorphisms. Such an action induces a diagonal action of $G$ on $\Sigma_3(X)$. If the action of $G$ on $\Sigma_3(X)$ is properly discontinuous and cocompact, we say $G$ acts on $X$ as a \emph{uniform convergence group}. 

Let $G$ be a group acting on a proper hyperbolic metric space by isometries. If $G$ acts on $\Lambda(G)$ as a uniform convergence group, then $G$ is a hyperbolic group and $\Lambda(G) = \Lambda_c(G)$. Readers may see Section $5$ of \cite{KB02} for a reference.

$\Gamma$ is a \emph{Kleinian group} if $\Gamma$ is a discrete subgroup of isometries of hyperbolic space $\mathbb{H}^n, n\geq 3$. Let $\textrm{Hull}(\Lambda(\Gamma)) \subset \mathbb{H}^n$ denotes the \emph{convex hull} of the limit set, i.e., the smallest convex set containing all geodesics with both endpoints in $\Lambda(\Gamma)$. Then $\Gamma$ acts \emph{convex cocompactly} on $\mathbb{H}^n$ if $\textrm{Hull}(\Lambda(\Gamma)) / \Gamma$ is compact. If $\Gamma$ is a convex cocompact Kleinian group, then $\Gamma$ acts on $\Lambda(\Gamma)$ as a uniform convergence group. 

Recall that $\mathbb{H}^n$ with its standard metric is a CAT($-1$) space. Theorem \ref{cls} directly induces the next result.

\begin{Corollary}\label{kge}
	Let $\Gamma$ be a Kleinian group that acts on $\Lambda(\Gamma)$ as a uniform convergence group. Then $\Lambda(\Gamma)$, equipped with any visual metric, is quasi-self-symmetric.
\end{Corollary}

\begin{Remark}
	Corollary \ref{kge} is a generalization of a classical result, i.e., Corollary $2.66$ on \cite{Ap00}.
\end{Remark}

The following theorem from \cite{Me14} illustrates a local rigidity result about the limit sets of Kleinian groups which are Schottky sets.

\begin{Theorem}[Merenkov]\label{lsss}
	Suppose that $\Gamma$ and $\tilde{\Gamma}$ are Kleinian groups whose limit sets $S$ and $\tilde{S}$ are Schottky sets, respectively. We assume that $\Gamma$ act on $S$ and $\tilde{\Gamma}$ act on $\tilde{S}$ as uniform convergence groups. If $f : A \to S$ is a quasiconformal embedding defined on an open (in relative topology) connected subset $A$ of $S$, then f has to be the restriction of a M\"obius transformation that takes $S$ onto $\tilde{S}$.
\end{Theorem}
Recall that a \emph{Schottky set} is a compact subset of $\mathbb{S}^2$ whose complement is a union of at least three open round discs whose closures have empty intersection. The original theorem in \cite{Me14} requires that the Kleinian groups act on their limit sets cocompactly on triples, but this is equivalent to that the group acting as a uniform convergence group. See \cite{GM87} for a reference. The original theorem also requires $f$ to be a quasiregular map, which is equivalent to a quasiconformal embedding.

Finally, we apply Thereon \ref{wttob} and Corollary \ref{kge} to $S$ and $\tilde{S}$ to induce a quasisymmetry from an open neighborhood of $S$ to an open neighborhood of $\tilde{S}$. Since every quasisymmetry is quasiconformal, this finises the proof of Theorem \ref{lsigm} by applying Theorem \ref{lsss}.

\section{Visual spheres of expanding Thurston maps}\label{VSETM}

An expanding Thurston map $f: S^2 \to S^2$ is a  branched covering map on a topological sphere that locally expands $S^2$. We investigate these maps in this section and apply our results to visual spheres of expanding Thurston maps. The main reference of this section is \cite{BM17}.

Let $S^2$ be a topological sphere and $f:S^2 \to S^2$ be a branched covering map on $S^2$ with $\deg(f) \geq 2$. A point $p \in S^2$ is a \emph{critical point} of $f$ if $f$ is not a local homeomorphism near $p$. We denote by $\textrm{crit}(f)$ the set of all critical points of $f$ and $f^n$ the $n$-th iteration of $f$. The \emph{postcritical points} of $f$ are given by
\[
\post(f) := \bigcup_{n \geq 1}\{ f^n(p): p \in \textrm{crit}(f) \}.
\]
We say $f$ is a \emph{Thurston map} if $f:S^2 \to S^2$ is a branched covering map such that $\deg(f) \geq 2$ and $\post(f)$ is finite.

Let $f : S^2 \to S^2$ be a Thurston map and $\mathcal{C}$ be a Jordan curve in $S^2$ with $\post(f) \subset \mathcal{C}$. We fix a base metric $d$ on $S^2$ that induces the given topology on $S^2$. The Jordan curve $\mathcal{C}$ divides $S^2$ into two parts and each of them(including the boundary) is called a \emph{$0$-tile}. The preimages of $0$-tiles under $f^n$ divide $S^2$ into what are called \emph{$n$-tiles}. For any $n \in \mathbb{N}$ we denote by $\textrm{mesh}(f, n, \mathcal{C})$ the supremum of the diameters of all $n$-tiles. A Thurston map $f : S^2 \to S^2$ is \emph{expanding} if there exists a Jordan curve $\mathcal{C} \subset S^2$ with $\post(f) \subset \mathcal{C}$ such that $\lim\limits_{n \to \infty} \textrm{mesh}(f, n, \mathcal{C}) = 0$. Notice that it is a topological property, i.e., it is independent of the choice of the base metric.

If $x \neq y$, we define
\[
m(x, y) = \max\{n \in \mathbb{N} \cup \{0\}: \exists \ \textrm{non-disjoint $n$-tiles} \ X \textrm{and} \ Y \ \textrm{such that} \ x \in X, \ y \in Y \}.
\]

If $x = y$, we define $m(x, y) := \infty$.

Let $f : S^2 \to S^2$ be a Thurston map. A metric $\rho$ on $S^2$ is called a \emph{visual metric} for $f$ if there exists a Jordan curve $\mathcal{C}$ in $S^2$ with $\post(f) \subset \mathcal{C}$, a parameter $\Lambda > 1$ and $C_1, C_2 > 0$ such that
\begin{equation}
C_1\Lambda^{-m(x,y)} \leq \rho(x,y) \leq C_2\Lambda^{-m(x,y)}
\end{equation}
for all $x, y \in S^2$. Here we define $\Lambda^{-\infty} := 0$. The number $\Lambda$ is called the \emph{expansion factor} of the visual metric and $C_1, C_2$ are independent of $x$ and $y$.

The following proposition illustrates some properties of the visual metrics on $S^2$. Readers may refer to Proposition $8.3$, Theorem $16.3$ and Theorem $18.1$ of \cite{BM17} for more details.

\begin{Proposition}\label{vmetm}
	Let $f : S^2 \to S^2$ be an expanding Thurston map, then
	\begin{enumerate}
		\item Every visual metric induces the standard topology on $S^2$.
		
		\item There exists a $\Lambda_0 > 1$ such that for any $\Lambda \in (1, \Lambda_0)$, there exists a visual metric on $S^2$ with expansion factor $\Lambda$. 
		
		\item Any two visual metrics are H\"older equivalent, and bi-Lipschitz equivalent if they have the same expansion factor $\Lambda$.
		
		\item Let $\rho$ be a visual metric for $f$. Then $(S^2, \rho)$ is doubling if and only if $f$ has no periodic critical points.

		\item A metric $\rho$ is a visual metric for some iterate $f^n$ with $n \in \mathbb{N}$ if and only if it is a visual metric for $f$.
	\end{enumerate}
\end{Proposition}

A visual sphere of an expanding Thurston map can be identified with the boundary at infinity of a certain Gromov hyperbolic space constructed from tiles. Readers may refer to Chapter $10$ of \cite{BM17} for more information.

The following proposition focuses on a specific visual metric $\rho_0$ constructed in Theorem $16.3$ of \cite {BM17}. Readers may refer to Theorem 1.0.5 of \cite{Wu19} for the proof.

\begin{Proposition}\label{vsqss}
	Let $f : S^2 \to S^2$ be an expanding Thurston map without periodic critical points. Then there exists a visual metric $\rho_0$ for $f$ and $\mathcal{C}$ such that the visual sphere $(S^2, \rho_0)$ is quasi-self-similar.
\end{Proposition}

\begin{Remark}
	If $\mathcal{C}$ is invariant under $f$, i.e., $f^n(\mathcal{C}) \subset \mathcal{C}$, then $\rho_0$ defined in Proposition \ref{vsqss} is a quasi-geodesic metric. Readers may refer to Chapter $16$ of \cite{BM17} for the construction.
\end{Remark}

We are prepared to prove Theorem \ref{vms} now. Notice that Wu also proved Theorem \ref{vms} in \cite{Wu19} with ideas from dynamics, here we give an alternative proof based on tools generated in this paper.

\begin{proof}[Proof of Theorem \ref{vms}] 
	\mbox{}
	
	$(1) \Longrightarrow (2)$: $S^2$ is doubling due to Proposition \ref{vmetm}; thus finishes the proof by Theorem \ref{QSweakT}.
	
	$(2) \Longrightarrow (3)$: It is trivial.
	
	$(3) \Longrightarrow (4)$: Notice that $(S, \rho)$ is quasi-self-symmetric by Proposition \ref{vmetm} and Proposition \ref{vsqss}. Let $T_pS^2$ be the weak tangent that is quasisymmetric to $\mathbb{R}^2$. Then there exists a $U \subset S^2$ which is quasisymmetrically embedded into $T_pX$ by Theorem \ref{wttob}.
	
	$(4) \Longrightarrow (1)$: Let $\rho_0$ be the visual metric defined in Proposition \ref{vsqss} where $\rho_0$ is constructed corresponding to $f$ and $\mathcal{C}$. Since any two visual metrics are quasisymmetric equivalent by Proposition \ref{vsqss}, it is sufficient to prove that $(S, \rho_0)$ is a quasi-sphere. We may assume that $\mathcal{C}$ is invariant under $f^{n}$ for some sufficiently large $n$ by Theorem $15.1$ in \cite{BM17}. Since $\rho_0$ is also a visual metric for $f^n$ and $\mathcal{C}$ by Proposition \ref{vmetm}, without loss of generality, we assume that in the following content $\mathcal{C}$ is invariant under $f$. 
	
	Let $U_1, U_2$ be the two $0$-tiles of $S^2$. Since $f$ is expanding, there exist two simply-connected closed subsets $U'_1, U'_2 \in U$ and $n \in \mathbb{N}$ such that $f^n : U'_1 \to U_1$ and $f^n: U'_2 \to U_2$ are homeomorphisms. Furthermore, they are also bi-Lipschitz by the definition of visual metrics. Thus there exist two quasisymmetric embeddings $f_1 : U_1 \to \mathbb{R}^2$ and $f_2 : U_2 \to \mathbb{R}^2$. Notice that $\mathcal{C}$ is a quasi-circle by Theorem $15.3$ in \cite{BM17}. Thus, by post-composing quasisymmetries, we may assume that $f_1 : U_1 \to \overline{\mathbb{D}}$ and $f_2 : U_2 \to \overline{\mathbb{D}}$ are two quasisymmetric maps. Here $\overline{\mathbb{D}}$ is the closed unit disk. Since $f_1 \circ f_2^{-1}$ is quasisymmetric on $\partial \mathbb{D}$, post-composing $f_2$ by a quasisymmetry again, we may assume that $f_1|_\mathcal{C} = f_2|_{\mathcal{C}}$. The quasisymmetries being post-composed above are constructed by Beurling-Ahlfors extension. Readers may refer to Section $5$ of \cite{Bo11} for a reference. Notice that $\overline{\mathbb{D}}$ is bi-Lipschitz to the closed half unit sphere with intrinsic metric. By gluing $f_1$ and $f_2$ via identifying the boundary, we obtain a homeomorphism $f: S^2 \to \mathbb{S}^2$. Since $S^2$ is a quasi-geodesic space and $\mathbb{S}$ is a length space, Applying Lemma \ref{gqs} and the Remarks following it finish the proof.
\end{proof}

\end{document}